\documentclass[reqno,10pt]{amsart}
\usepackage{amsmath}	
\usepackage{amssymb}	
\usepackage{amsthm}
\usepackage{amsfonts}
\usepackage{float}
\usepackage{latexsym}
\usepackage{mathrsfs}
\usepackage{tabularx}
\usepackage{pict2e}
\usepackage{nicematrix}
\usepackage{graphicx}
\usepackage[all]{xy}
\usepackage{tikz-cd}
\usepackage{tikz}
\usetikzlibrary{matrix,arrows}
\usepackage[colorlinks=true, citecolor=blue, urlcolor=blue, linkcolor=blue]{hyperref}

\theoremstyle{plain} 
\newtheorem{thm}{Theorem}[section]
\newtheorem{lem}[thm]{Lemma} 
\newtheorem{prop}[thm]{Proposition} 
\newtheorem{cor}[thm]{Corollary} 

\newtheorem*{prob}{Problem A}
\theoremstyle{definition} 
\newtheorem{defn}[thm]{Definition}
\newtheorem{rem}[thm]{Remark} 
\newtheorem{ex}[thm]{Example}

\makeatletter
\@namedef{subjclassname@2020}{%
  \textup{2020} Mathematics Subject Classification}
\makeatother

\newcommand{\A}{\mathcal{A}}

\newcommand{\Q}{\mathbb{Q}}
\newcommand{\ZN}{\mathbb{Z}}
\newcommand{\RN}{\mathbb{R}}
\newcommand{\PB}{\mathbb{P}}
\newcommand{\CO}{\mathcal{O}}
\newcommand{\CN}{\mathbb{C}}
\newcommand{\CC}{{\rm C}}

\newcommand{\DC}{\mathrm{D}^{{\rm b}}}
\newcommand{\Hom}{\mathrm{Hom}}

\newcommand{\Ext}{\mathrm{Ext}}
\newcommand{\ext}{\mathrm{ext}}
\newcommand{\ch}{\mathrm{ch}}

\newcommand{\Coh}{\mathrm{Coh}}
\newcommand{\Stab}{\mathrm{Stab}}

\newcommand{\RCH}{R\mathcal{H}om}

\newcommand{\CH}{\mathcal{H}}
\newcommand{\Hilb}{\mathrm{Hilb}}
\newcommand{\MM}{\mathcal{M}}

\allowdisplaybreaks
\numberwithin{equation}{section}

\begin{document}

\title{Quintic genus 2 curves, stable pairs and wall-crossing}

\author{Shihao Ma \and Song Yang}
\address{Center for Applied Mathematics and KL-AAGDM, Tianjin University, Weijin Road 92, Tianjin 300072, P. R. China}%
\email{shma@tju.edu.cn, syangmath@tju.edu.cn}%

\begin{abstract}
This paper studies wall crossings in Bridgeland stability for the moduli space of Pandharipande--Thomas stable pairs associated with quintic genus 2 curves in the complex projective three-space. We provide a complete list of irreducible components of the moduli space, together with a birational description of each component. As an application, we recover the geometry of the Hilbert scheme for quintic genus 2 curves, for which no previous results are known.
\end{abstract}

\date{\today}


\maketitle


\section{Introduction}
The theory of curve-counting on a smooth complex projective $3$-fold $X$, introduced by Pandharipande--Thomas \cite{PT09} via the moduli spaces of stable pairs, has been extensively studied over the two decades.
A stable pair $(F,s)$ on $X$ consists of a $1$-dimensional pure sheaf $F$ and a non-trivial morphism $s:\CO_{X} \rightarrow F$ with $0$-dimensional cokernel. 
The moduli space $P_{n}(X,\beta)$ parameterizes such pairs with the Euler characteristic and the homology class of the support of $F$, $\chi(F)=n\in \ZN$ and $[F]=\beta\in H_{2}(X,\ZN)$.
The resulting moduli space provides a sheaf-theoretic compactification of the space of smooth curves in $X$ and carries an algebraic virtual class defining the Pandharipande--Thomas invariants. 
However, the global geometry of this fine projective moduli space remains poorly understood in general. 
In particular, we denote by $P_{1-g}(\PB^{3},d[\ell])$ the moduli space of stable pairs for curves of degree $d$ and arithmetic genus $g$ in the complex projective three-space $\PB^{3}$,
where $[\ell]$ is the homology class of a line $\ell\subset \PB^{3}$.
This paper focuses on the following fundamental problem that is still widely open.

\begin{prob}
What is the global geometry of the moduli space $P_{1-g}(\mathbb{P}^{3},d[\ell])$?    
\end{prob}

In the complex projective three-space, the global geometry of a given moduli space of stable pairs is generally very difficult to study, much like that of Hilbert schemes of curves. 
A famous theorem of Castelnuovo states that for a non-degenerate smooth curve of degree $d\geq 3$ in $\PB^{3}$,
its genus $g\leq  \frac{1}{4}(d^2-1)-d+1$ (resp. $\frac{1}{4}d^2-d+1$) if $d$ is odd (resp. even). 
The first four cases with maximal genus are $(d,g)=(3,0)$, $(4,1)$, $(5,2)$ and $(6,4)$.
For the Hilbert schemes, only a handful of general results are known, and very few examples are fully understood; 
see the maximal cases, for example,  \cite{PS85,EPS87} for twisted cubics, \cite{AV92,CN12} for elliptic quartics, and \cite{Rez24a} for canonical genus $4$ curves.
Results concerning the moduli space of stable pairs are even scarce. 
The first non-trivial case is the moduli space $P_{1}(\PB^{3},3[\ell])$ for twisted cubics. 
Based upon the work \cite{PS85}, it was observed in \cite{CCM16} that $P_{1}(\PB^{3},3[\ell])$ is isomorphic to the moduli space of Gieseker semistable sheaves with the Hilbert polynomial $3t+1$. 
This Gieseker moduli space is well-understood, largely due to the work \cite{FT04}.
The study of these classical moduli spaces has been revolutionized by the notion of stability conditions on triangulated categories introduced by Bridgeland \cite{Bri07}.
The existence of Bridgeland stability conditions on the bounded derived category of coherent sheaves $\DC(\PB^{3})$ was obtained in \cite{BMT14,Mac14b}.
Let $\sigma$ be a stability condition on $\DC(\PB^{3})$ and $M_{\sigma}(v)$ denote the moduli space of $\sigma$-semistable objects with the Chern character $v$.
In stability conditions, 
there exists a wall-chamber structure such that $M_{\sigma}(v)$ is constant if the stability condition $\sigma$ lies in a chamber, but may change if $\sigma$ crosses a wall. 
As pointed out by Macr{\`{\i}} \cite{Mac14b}: {\it it would be very interesting to study how moduli spaces of Bridgeland semistable objects vary when varying the stability condition}.
The first examples of wall crossings for Hilbert schemes associated with curves of maximal genus in $\PB^{3}$ appeared in \cite{Sch20a,Xia18} for twisted cubics and \cite{GHS18} for elliptic quartics; see also \cite{SG24} for skew lines.
Rezaee \cite{Rez24a,Rez24b} initiated the study of wall crossings in Bridgeland stability for the moduli space of stable pairs of canonical genus $4$ curves, 
provided the first full list of irreducible components with birational descriptions and a partial list for the Hilbert scheme.
Nevertheless, the global geometries of the Hilbert scheme and the moduli space of stable pairs for smooth curves of degree $5$ and genus $2$ remain open.

The purpose of this paper is to study the global geometry of the moduli space of stable pairs $P_{-1}(\PB^{3},5[\ell])$ for {\it smooth curves of degree $5$ and genus $2$} or {\it quintic genus $2$ curves} for short, via wall crossings in Bridgeland stability.
Roughly speaking, for a quintic genus $2$ curve $C$, 
we have $v:=\ch(I_{C})=(1,0,-5,11)$, the Chern character of the ideal sheaf of $C$. 
Then, for a stability condition $\sigma$ on $\DC(\PB^{3})$, we have the moduli space $M_{\sigma}(v)$ of $\sigma$-semistable objects $E\in \DC(\PB^{3})$ with $\ch(E)=v$.
The goal is to investigate how $M_{\sigma}(v)$ changes via wall-crossing, along a special path $\gamma$ in the space of stability conditions (see Figure \ref{main-wall-crossing-picture}).
This path $\gamma$ crosses five chambers with four walls given by explicit destabilizing exact sequences. 
We denote by $\mathcal{M}_{0}, \mathcal{M}_{1}, \cdots, \mathcal{M}_{4}$ the Bridgeland moduli spaces with Chern character $v$ for the chambers from inside the smallest wall to outside the largest wall.
Since the BMT inequality gives a semicircle in the first chamber, the moduli space $\mathcal{M}_{0}$ is empty.
For the second chamber, the moduli space $\mathcal{M}_{1}$ gives the main component a $\mathbb{P}^{11}$-bundle over a projective variety of Fano type; this provides an efficient compactification of the space of quintic genus $2$ curves. 
In the large volume limit, the moduli space $\mathcal{M}_{4}$  is isomorphic to $P_{-1}(\PB^{3},5[\ell])$.
More precisely, the main result of this paper is stated as follows.

\begin{thm}\label{mainthm}
The moduli space $P_{-1}(\PB^{3},5[\ell])$ consists of six irreducible components which are birational to the following components:
\begin{enumerate}
\item[(1)] A $20$-dimensional $\mathbb{P}^{11}$-bundle over a projective variety $\mathcal{K}_{4}(2,2)$ of Fano type, which generically parameterizes quintic genus $2$ curves;
\item[(2)] A $21$-dimensional $\mathbb{P}^{12}$-bundle over $\mathrm{Gr}(2,4)\times \mathfrak{Fl}_{1}$, which  generically parameterizes the union of a line and a plane quartic curve together with a choice of one point on the quartic;
\item[(3)] A $21$-dimensional $\mathbb{P}^{13}$-bundle over $\mathcal{U}\times(\PB^{3})^{\vee}$, which generically parameterizes the union of a line and a plane quartic intersecting the line; 
\item[(4)] A $21$-dimensional $\mathbb{P}^{14}$-bundle over $\mathrm{Gr}(2,4)\times (\PB^{3})^{\vee}$, which generically parameterizes the disjoint union of a line and a plane quartic;
\item[(5)] A $20$-dimensional $\mathbb{P}^{14}$-bundle over $\mathrm{Fl}(4)$, which generically parameterizes the union of a plane cubic with a thickening of a line in the plane;
\item[(6)] A $27$-dimensional $\mathbb{P}^{16}$-bundle over $\mathfrak{Fl}_{4}$, which generically parameterizes a plane quintic curve together with a choice of four points on it.
\end{enumerate}
Here, $\mathcal{K}_{4}(2,2)$ is the King moduli space of the quiver representations of the Kronecker quiver $K_{4}$ with the dimension vector $(2,2)$, 
$\mathfrak{Fl}_{i}$ is the space parameterizing flag $Z_{i}\subset V\subset \PB^{3}$ with $V$ a plane and $Z_{i}$ a $0$-dimensional subscheme of length $i$, $\mathcal{U}$ is the universal line over $\mathrm{Gr}(2,4)$, 
and $\mathrm{Fl}(4)$ is the flag manifold of $\CN^{4}$.
\end{thm}

As an application of Theorem \ref{mainthm}, we study the geometry of the Hilbert scheme $\Hilb^{5t-1}(\PB^{3})$.
In the large volume limit, by roughly analyzing wall crossings, we obtain some previously unknown results on the geometry of the Hilbert scheme.
Let $\mathcal{H}^{0}_{5,2}$ be the open subscheme of the Hilbert scheme $\mathrm{Hilb}^{5t-1}(\mathbb{P}^{3})$, parameterizing quintic genus 2 curves in $\mathbb{P}^{3}$. 
It is well-known that $\mathcal{H}^{0}_{5,2}$ is irreducible of dimension $20$ (see e.g. \cite[\S 19.4]{EH24}).
However, to the best of our knowledge, what is the closure $\overline{\mathcal{H}^{0}_{5,2}}\subset \mathrm{Hilb}^{5t-1}(\mathbb{P}^{3})$ in the whole Hilbert scheme—remains an open question (see e.g. \cite[Chapter 19]{EH24}).
In fact, we have the following:

\begin{cor}\label{main-cor}
The Hilbert scheme $\Hilb^{5t-1}(\PB^{3})$ has components birational to:
\begin{enumerate}
\item[(i)] A $20$-dimensional $\mathbb{P}^{11}$-bundle over a projective variety $\mathcal{K}_{4}(2,2)$ of Fano type, which generically parameterizes quintic genus $2$ curves;
\item[(ii)] A $21$-dimensional $\mathbb{P}^{13}$-bundle over $\mathcal{U}\times(\PB^{3})^{\vee}$, which generically parameterizes the union of a line and a plane quartic intersecting the line; 
\item[(iii)] A $21$-dimensional $\mathbb{P}^{14}$-bundle over $\mathrm{Gr}(2,4)\times (\PB^{3})^{\vee}$, which generically parameterizes the disjoint union of a line and a plane quartic;
\item[(iv)] A $20$-dimensional $\mathbb{P}^{14}$-bundle over $\mathrm{Fl}(4)$, which generically parameterizes the union of a plane cubic with a thickening of a line in the plane.
\end{enumerate}
In particular, the main irreducible component $\overline{\mathcal{H}^{0}_{5,2}}$  is birational to a $\mathbb{P}^{11}$-bundle over $\mathcal{K}_{4}(2,2)$.
\end{cor}

The structure of this paper is organized as follows. 
In Section \ref{Prelim}, we collect some basics on tilt stability, Bridgeland stability conditions, wall and chamber structures and moduli spaces of Pandharipande--Thomas stable pairs.
Section \ref{Wall-chamber} gives the geometric description of the walls in tilt stability and Bridgeland stability with respect to the Chern character $v=(1,0,-5,11)$ of the ideal sheaf of a quintic genus $2$ curve.
Section \ref{Ext-on-walls} calculates all relevant Ext groups related to these walls.
Theorem \ref{mainthm} and Corollary \ref{main-cor} are proved in Section \ref{proofs-main}.
In Appendix \ref{three-tech-lem}, three technical lemmas are provided.

\subsection*{Acknowledgements}
This work is partially supported by the National Natural Science Foundation of China (No. 12171351).

\section*{Notation and conventions}
Throughout the paper, we work over the complex number field $\CN$.
Let $X$ be a smooth projective variety and all points are close points.
We use $\DC(X):=\DC(\Coh(X))$ to denote the bounded derived category of coherent sheaves.
All functors are derived functors without being explicitly mentioned; for instance, the direct image functor, inverse image functor and tensor product.
Without causing confusion, we use $A$ and $B$ interchangeably to denote the subobject and the quotient in the destabilizing exact sequence of any wall, for both the tilt stability $\sigma_{\alpha,\beta}$ and the Bridgeland stability $\lambda_{\alpha,\beta,s}$. 
We use $\langle A, B \rangle$ to denote the destabilizing pair or the corresponding wall.
We will choose a special path crossing the walls toward the large volume limit (see Figure \ref{main-wall-crossing-picture}). 
By the large volume limit, we mean the region in the stability manifold of $\DC(X)$, where $\alpha\rightarrow +\infty$. 
We denote by $\Gamma$ the left branch of the hyperbola $\beta^{2}-\alpha^{2}=10$ in the $(\alpha,\beta)$-plane.
By closing $\Gamma$, we mean a sufficiently small tubular neighborhood of
the locus of $\Gamma$, so that there is no other wall intersecting this neighborhood.
We will use the following notation.
\begin{center}
\begin{tabular}{r l }
$\CH^{i}$ & The $i$-th cohomology object in the corresponding heart on $\DC(X)$\\
$H^{i}(E)$ & The $i$-th sheaf cohomology group of object $E\in \DC(X)$\\
$\ch_{i}(E)$ & The $i$-th Chern character of an object $E\in \DC(X)$\\
$\ch_{\leq i}(E)$ & The truncated Chern character $(\ch_{0}(E),\cdots,\ch_{i}(E))$\\
$W(c,r)$ & The semicircle defined by $\{(\alpha,\beta)\in \RN_{>0}\times \RN | \alpha^{2}+(\beta-c)^{2}=r^{2}\}$\\
$\mathcal{K}_{4}(2,2)$ & The King moduli space of the quiver representations of the Kronecker \\& quiver $K_{4}$ with the dimension vector $(2,2)$\\ 
$\mathfrak{Fl}_{i}$ & The space parameterizing flag $Z_{i}\subset V\subset \PB^{3}$ with $V$ a plane and $Z_{i}$ a\\& $0$-dimensional subscheme of length $i$\\
$\mathrm{Fl}(4)$ & The flag manifold of $\CN^{4}$\\
$\mathcal{U}$ & The universal line over the Grassmann manifold $\mathrm{Gr}(2,4)$
\end{tabular}
\end{center}

 
\section{Preliminaries}\label{Prelim}

In this section, we collect some basics on tilt stability, Bridgeland stability conditions, wall and chamber structures and moduli spaces of Pandharipande--Thomas stable pairs.

\subsection{Tilt stability}

Let $X$ be a smooth projective $3$-fold. Denote by $K(X)$ the Grothendieck group of $\DC(X)$. 
Fix a lattice $\Lambda$ of finite rank and a surjective homomorphism $v:K(X)\rightarrow \Lambda$.
Suppose that $\A$ is the heart of a bounded $t$-structure on $\DC(X)$.
A {\it (weak) stability function} on $\A$ is a group homomorphism (called a {\it central charge}) $Z: \Lambda \rightarrow \CN$ such that for every non-zero object $E\in \A$, we have
$\mathrm{Im}\, Z(v(E))\geq 0$, and if $\mathrm{Im}\, Z(v(E))= 0$, then $\mathrm{Re}\, Z(v(E))\, (\leq ) < 0$.
For an object $E\in \A$, we set $Z(E):=Z(v(E))$ for simplicity.
The {\it $\mu_{\sigma}$-slope} of a  non-trivial object $E\in \A$ is defined as 
$$
\mu_{\sigma}(E)
:=   
\begin{cases}
-\frac{\mathrm{Re}\, Z(E)}{\mathrm{Im}\, Z(E)} & \textrm{if}\, \mathrm{Im}\, Z(E)>0; \\
+\infty &  \textrm{otherwise,}
\end{cases}
$$
where $\sigma:=(\A,Z)$.
An object $E\in \A$ is called {\it $\sigma$-(semi)stable} if for every non-trivial proper subobject $F\subset E$, we have $\mu_{\sigma}(F) \, (\leq) < \mu_{\sigma}(E/F)$.

\begin{defn}
Let $Z$ be a (weak) stability function on $\A$.
We say that $\sigma=(\A,Z)$ is a {\it (weak) stability condition} on $\DC(X)$ (with respect to $\Lambda$) if the following conditions hold:
\begin{enumerate}
\item[(i)] Every non-zero object $E\in \A$ has a Harder--Narasimhan filtration in $\sigma$-semistability.

\item[(ii)] There is a quadratic form $\Delta$ on $\Lambda\otimes \mathbb{R}$ such that $\Delta|_{\ker(Z)}$ is negative definite and $\Delta(E)\geq 0$ for every $\sigma$-semistable object $E\in \A$.
\end{enumerate}
\end{defn}

\begin{rem} 
The set $\Stab_{\Lambda}(\DC(X))$ of stability conditions has a complex manifold structure by \cite[Theorem 1.2]{Bri07}, which is called the stability manifold of $\DC(X)$.  
If $\Lambda=K_{num}(X)$ is the numerical Grothendieck group, then we drop the subscript $\Lambda$.
\end{rem}

Assume that $\sigma=(\A, Z)$ is a weak stability condition on $\DC(X)$.
For a real number $\mu$, 
there exists a torsion pair $(\mathcal{T}_{\sigma}^{\mu}, \mathcal{F}_{\sigma}^{\mu})$ on $\A$: 
\begin{align*} 
\mathcal{T}_{\sigma}^{\mu} &:=
\{E \in \A \mid \textrm{ any quotient $E\twoheadrightarrow G$ satisfies } \mu_{\sigma}(G)>\mu\}, \\
\mathcal{F}_{\sigma}^{\mu} &:= 
\{ E \in \A \mid \textrm{ any subsheaf $F\hookrightarrow R$  satisfies } \mu_{\sigma}(F)\leq \mu\}.
\end{align*}
Then, the heart obtained by tilting $\A$ with respect to $\sigma$ at the slope $\mu$ is given by the extension closure
\begin{align*}
\A^{\mu}_{\sigma}
&:=
\langle \mathcal{T}_{\sigma}^{\mu}, \mathcal{F}_{\sigma}^{\mu}[1]\rangle \\
&= \{E \mid \CH_{\A}^{0}(E)\in \mathcal{T}_{\sigma}^{\mu}, \CH_{\A}^{-1}(E)\in \mathcal{F}_{\sigma}^{\mu}, \CH_{\A}^{i}(E)=0\, \textrm{ for } i\neq 0, -1 \}.
\end{align*}

From now on, let $H$ be an ample divisor on $X$.
For $\PB^{3}$, we use the hyperplane $H$ without explicitly mentioning it.
For every integer $j\in \{0,1, 2, 3\}$, 
we consider the lattice $\Lambda_{H}^{j}\cong \ZN^{j+1}$ generated by 
$
(H^{3}\ch_{0}, H^{2}\ch_{1}, \cdots, H^{3-j}\ch_{j})\in \Q^{j+1}
$
with the map $\nu_{H}^{j}: K(X)\rightarrow \Lambda_{H}^{j}$.
Then, the pair $\sigma_{H}:= (\Coh(X), Z_{H})$ with the central charge
$
Z_{H}(E):=  -H^{2} \cdot \ch_{1}(E)+ i H^{3} \cdot \ch_{0}(E)
$
is a weak stability condition on $\DC(X)$ with respect to the lattice $\Lambda^{1}_{H}$, which is known as {\it $\mu_{H}$-slope stability}.
For a real number $\beta$,
we denote by 
$$
\Coh^{\beta}(X):=\langle  \mathcal{T}_{\sigma_{H}}^{\beta}, \mathcal{F}_{\sigma_{H}}^{\beta}[1]\rangle
$$ 
the heart of a bounded $t$-structure on $\DC(X)$ obtained by tilting $\Coh(X)$ with respect to the $\mu_{H}$-slope stability at the slope $\beta$.
We use $\CH_{\beta}^{i}$ to denote the $i$-th cohomology object with respect to the heart
$\Coh^{\beta}(X)$.
For any object $E\in \DC(X)$, the twisted Chern character is given by $\ch^{\beta}(E):=e^{-\beta H}\cdot \ch(E)$.
Then a new central charge on $\Coh^{\beta}(X)$ is defined as
$$
Z_{\alpha, \beta}(E)
:=\frac{1}{2}\alpha^2 H^{3}\cdot \ch_{0}^{\beta}(E)-H \cdot \ch_{2}^{\beta}(E)
+ i H^{2}\cdot \ch_{1}^{\beta}(E), 
$$
where $(\alpha,\beta)\in \RN_{>0}\times \RN$.

\begin{prop}[{\cite[Proposition 2.12]{BLMS23}}]
There exists a continuous family of weak stability conditions (with respect to $\Lambda_{H}^{2}$), 
$\sigma_{\alpha, \beta}:= (\Coh^{\beta}(X), Z_{\alpha, \beta})$, 
parameterizing $(\alpha, \beta)\in \RN_{>0}\times \RN$ with the quadratic form given by the $H$-discriminant $\Delta_{H}$. Here, for any $E\in \DC(X)$, the $H$-discriminant $\Delta_{H}$ is defined as
\begin{align*}
\Delta_{H}(E)
&:= (H^{2}\cdot \ch^{\beta}_{1}(E))^{2}-2(H^{3}\cdot\ch^{\beta}_{0}(E))\cdot (H \cdot \ch^{\beta}_{2}(E)) \\ 
&= (H^{2}\cdot \ch_{1}(E))^{2}-2(H^{3}\cdot\ch_{0}(E))\cdot (H \cdot \ch_{2}(E)).
\end{align*}
\end{prop}

The weak stability condition $\sigma_{\alpha, \beta}$ is also known as {\it tilt stability} introduced by Bayer--Macr{\`{\i}}--Toda \cite{BMT14}. 
The {\it tilt slope} is given by
$$
\mu_{\alpha,\beta}(E):=\mu_{\sigma_{\alpha,\beta}}(E)
=-\frac{\frac{1}{2}\alpha^2 H^{3}\cdot \ch_{0}^{\beta}(E)-H \cdot \ch_{2}^{\beta}(E)}{H^{2}\cdot \ch_{1}^{\beta}(E)},
$$
if $H^{2}\cdot \ch_{1}^{\beta}(E)\neq 0$; otherwise, $\mu_{\alpha,\beta}(E):=+\infty$.
Moreover, there exists a version of the Bogomolov inequality for tilt stability (see \cite[Corollary 7.3.2]{BMT14} and \cite[Theorem 3.5]{BMS16}):

\begin{prop}\label{tilt-Bog-ineq}
If $E\in \Coh^{\beta}(X)$ is $\sigma_{\alpha,\beta}$-semistable,
then $\Delta_{H}(E)\geq 0$.
\end{prop}

The following large volume limit plays a significant role in tilt stability (see \cite[Proposition 14.2]{Bri08}, \cite[Proposition 4.8]{BBF+} and \cite[Lemma 2.7]{BMS16}).

\begin{prop}\label{limit-tilt-stab-Gie-stab-prop}
Let $E\in \DC(X)$ and $\mu_{H}(E)>\beta$.
Then $E\in \Coh^{\beta}(X)$ and $E$ is $\sigma_{\alpha,\beta}$-(semi)stable for $\alpha\gg 0$
if and only if $E$ is a $2$-$H$-Gieseker (semi)stable sheaf.
\end{prop}

\begin{defn}\label{numerical-wall-def}
Let $v\in K(X)$.
\begin{enumerate}
  \item[(1)] A {\it numerical wall} for $v$ with respect to $w\in K(X)$ in tilt stability is a non-trivial proper subset in the upper half plane
$$
W(v,w):=  \{ (\alpha, \beta)\in \RN_{>0}\times \RN \mid \mu_{\alpha,\beta}(v)=\mu_{\alpha,\beta}(w)\}.
$$ 
  \item[(2)] A {\it chamber} for $v$ is a connected component in the complement of the union of numerical walls in the upper half plane. 
\end{enumerate}
\end{defn}

The numerical walls in tilt stability have a well-behaved wall and chamber structure (see e.g. \cite{Mac14a} for surfaces and \cite[Theorem 3.3]{Sch20a} for $3$-folds). 

\begin{thm}[]\label{tilt-wall-struct-thm}
Let $v \in K(X)$ with $\Delta_{H}(v) \geq 0$.
All numerical walls for $v$ in the $(\alpha,\beta)$-plane are as follows:
\begin{enumerate}
\item[(1)] A numerical wall for $v$ is either a semicircle centered at $\beta$-axis or a vertical line parallel to $\alpha$-axis.
\item[(2)] If $\ch_{0}(v) \neq 0$, then there is a unique vertical wall for $v$ given by $\beta=\mu_{H}(v)$. 
The remaining numerical walls for $v$ consist of two sets of nested semicircular walls whose apexes lie on the hyperbola $\mu_{\alpha,\beta}(v)=0$. 
\item[(3)] If $\ch_{0}(v) = 0$ and $H^{2} \cdot \ch_{1}(v) \neq 0$, 
then every numerical wall for $v$ is a semicircle whose top point lies on the vertical line $\beta=\frac{H \cdot \ch_{2}(v)}{H^{2} \cdot \ch_{1}(v)}$.
\end{enumerate}
\end{thm}

One of the key techniques and the main difficult is to determine actual walls in tilt stability.

\begin{defn}\label{wall-def}
A numerical wall $W$ for $v\in K(X)$ is called an {\it actual wall} for $v$ if there exists a short exact sequence of $\sigma_{\alpha,\beta}$-semistable objects
\begin{equation}\label{def-wall-sequ}
\xymatrix@C=0.5cm{
0 \ar[r]^{} & A \ar[r]^{} & E \ar[r]^{} & B \ar[r]^{} & 0}
\end{equation}
in $\Coh^{\beta}(X)$ for one $(\alpha,\beta)\in W(E,A)$ such that $v=\ch(E)$ and $\mu_{\alpha,\beta}(A)=\mu_{\alpha,\beta}(E)$ define the numerical wall $W$ (i.e., $W=W(E,A)$). 
In particular, \eqref{def-wall-sequ} is called a {\it destabilizing exact sequence} and the pair $\langle A,B \rangle$ is called a {\it destabilizing pair}.
\end{defn}

\subsection{Bridgeland stability conditions}
In this subsection, we always assume that $X$ is a Fano $3$-fold of Picard number $1$.
To construct Bridgeland stability conditions on $\DC(X)$, we need the second tilting from tilt stability:
\begin{align*}
\mathcal{T}_{\sigma_{\alpha,\beta}}^{0} &:= 
\{ E \in \text{Coh}^{\beta}(X) \mid  \textrm{ any quotient $E\twoheadrightarrow G$ satisfies } \mu_{\alpha, \beta}(G) > 0 \}, \\
\mathcal{F}_{\sigma_{\alpha,\beta}}^{0} &:=
\{ E \in \text{Coh}^{\beta}(X) \mid \text{ for all subsheaf $F \hookrightarrow E$ satisfies } \mu_{\alpha, \beta}(F) \leq 0 \}. 
\end{align*}
The corresponding {\it tilted heart} is given by the extension closure
$$
\Coh^{\alpha,\beta}(X):=
\langle  \mathcal{T}_{\sigma_{\alpha,\beta}}^{0} , \mathcal{F}_{\sigma_{\alpha,\beta}}^{0}[1]\rangle.
$$
The new {\it central charge} is defined as 
$$
Z_{\alpha,\beta,s}(E)
:=-\ch_{3}^{\beta}(E)+(s+\frac{1}{6})\alpha^{2} H^{2}\cdot \ch_{1}^{\beta}(E)
+ i \big(H\cdot \ch_{2}^{\beta}(E)-\frac{1}{2}\alpha^{2} H^{3}\cdot \ch_{0}^{\beta}(E) \big),
$$ 
where $E\in \Coh^{\alpha,\beta}(X)$ and $s>0$. 
The {\it slope} is given by
$$
\lambda_{\alpha,\beta,s}(E)
:=\begin{cases}
-\frac{-\ch_{3}^{\beta}(E)+(s+\frac{1}{6})\alpha^{2} H^{2}\cdot \ch_{1}^{\beta}(E)}{H\cdot \ch_{2}^{\beta}(E)-\frac{1}{2}\alpha^{2} H^{3}\cdot \ch_{0}^{\beta}(E)}, & \mathrm{Im}\, (Z_{\alpha,\beta,s}(E))\neq 0\\
+\infty, & \mathrm{Im}\,(Z_{\alpha,\beta,s}(E))=0.
\end{cases}
$$
The following Bogomolov-type inequality (always called {\it BMT inequality}) is the key to show that the pair $(\Coh^{\alpha,\beta}(X),Z_{\alpha,\beta,s})$ is a Bridgeland stability condition on $\DC(X)$.

\begin{thm}\label{BMT-inequality}
If $E\in \DC(X)$ is a $\sigma_{\alpha,\beta}$-semistable object,
then
$$
Q_{\alpha,\beta}(E):=
\frac{\alpha^{2}+\beta^{2}}{2} \big( \CC_{1}^{2}-2\CC_{0}\CC_{2}\big)+\beta \big(3\CC_{0}\CC_{3}-\CC_{1}\CC_{2}\big)+\big(2\CC_{2}^{2}-3\CC_{1}\CC_{3}\big) \geq 0,
$$
where $\CC_{i}:=H^{3-i}\cdot \ch_{i}(E)$ for $i\in \{0,1,2,3\}$.
\end{thm}  

\begin{proof}
See \cite[Theorem 1.1]{Mac14b} for $\PB^{3}$, 
\cite{Sch14} for quadric $3$-fold and \cite{Li19} for the general cases.
\end{proof}

Under the BMT inequality,
the existence of Bridgeland stability conditions on $\DC(X)$
was established in \cite{BMT14};
see \cite[Section 8]{BMS16} for more detailed discussions.

\begin{thm}[{\cite{BMT14,BMS16}}]
For any $s>0$, the pair $(\Coh^{\alpha,\beta}(X),Z_{\alpha,\beta,s})$ is a Bridgeland stability condition with the support property.
Moreover, the function 
$$
\begin{array}{cccl}
& \RN_{>0}\times \RN\times\RN_{>0}  &\longrightarrow& \Stab(\DC(X)) \\
&(\alpha,\beta,s) &\longmapsto&  (\Coh^{\alpha,\beta}(X),Z_{\alpha,\beta,s})
\end{array}
$$
is continuous. 
\end{thm}

The existence for the moduli spaces of Bridgeland semistable objects is guaranteed by the following result.

\begin{thm}
Let $\sigma$ be a Bridgeland stability condition on $\DC(X)$.
Then the moduli stack $\mathcal{M}_{\sigma}(v)$ of $\sigma$-semistable objects with Chern character $v$ is an algebraic stack of finite type over $\CN$.
If there exist no strictly semistable objects, 
then $\mathcal{M}_{\sigma}(v)$ is a $\mathbb{G}_{m}$-gerbe over its coarse moduli space $M_{\sigma}(v)$, which is a proper algebraic space over $\CN$.
\end{thm}

\begin{proof}
See \cite[Theorem 1.2]{PT19} for smooth projective $3$-folds satisfying the BMT inequality and \cite[Theorem 21.24]{BLM+21} for general cases.
\end{proof}

Analogous to tilt stability,
the notions of the numerical wall, actual wall and chamber in Bridgeland stability conditions are defined as follows:

\begin{defn}
Let $v,w\in K(X)$ be two non-parallel classes.
\begin{enumerate}
\item[(1)] A {\it numerical wall} of $v$ in Bridgeland stability with respect to $w$  is a non-trivial proper subset of the stability space, which is defined as
$$
W_{s}(v,w)
=\{(\alpha,\beta)\in \RN_{>0}\times \RN \mid \lambda_{\alpha,\beta,s}(v)=\lambda_{\alpha,\beta,s}(w)\};
$$
\item[(2)] An {\it actual wall} of $v$ is the subset of a numerical wall consisting of those points for which there exist properly semistable objects with class $v$;
\item[(3)] A {\it chamber} is defined to be a connected component of the complement of the set of actual walls.
\end{enumerate}
\end{defn}

The following result indicates a close relationship between tilt stability $\sigma_{\alpha,\beta}$ and Bridgeland stability $\lambda_{\alpha,\beta,s}$.

\begin{lem}[{\cite[Lemma 8.9]{BMS16},\cite[Lemma 3.5]{Rez24b}}]\label{tilt-stab-Bri-stab}
Let $E\in \Coh^{\alpha,\beta}(X)$ be $\lambda_{\alpha,\beta,s}$-semistable for all $s\gg 1$ sufficiently big.
Then one of the following conditions holds:
\begin{enumerate}
  \item[(i)] $E\cong \CH_{\beta}^{0}(E)$ is $\sigma_{\alpha,\beta}$-semistable;
  \item[(ii)] $ \CH_{\beta}^{-1}(E)$ is $\sigma_{\alpha,\beta}$-semistable and $ \CH_{\beta}^{0}(E)$ is supported in dimension $0$.
\end{enumerate}
If moreover, $\CH_{\beta}^{-1}(E)$ is $\sigma_{\alpha,\beta}$-stable, 
$ \CH_{\beta}^{0}(E)$ is a torsion sheaf supported in dimension $0$, 
and $\Hom(\CO_{x},E)=0$ for all $x\in X$, 
then $E$ is $\lambda_{\alpha,\beta,s}$-stable for all $s\gg 1$.
\end{lem}

In particular,  
by Lemma \ref{tilt-stab-Bri-stab},  the following result holds.

\begin{cor}\label{tilt-stab-Bri-stab-cor}
\begin{enumerate}
\item[(1)] If $E$ is $\lambda_{\alpha,\beta,s}$-semistable with $\mu_{\alpha,\beta}(E)>0$ and $\ch_{1}^{\beta}(E)>0$, then $\CH^{-1}_{\beta}(E)=0$ and $\CH^{0}_{\beta}(E)$ is $\sigma_{\alpha,\beta}$-semistable.
\item[(2)] If $E$ is $\lambda_{\alpha,\beta,s}$-semistable with $\mu_{\alpha,\beta}(E)<0$ and $\ch_{1}^{\beta}(E)>0$, then $\CH^{0}_{\beta}(E)$ is $\sigma_{\alpha,\beta}$-semistable and $\CH^{1}_{\beta}(E)$ is 0 or a torsion sheaf with $0$-dimensional support.
\end{enumerate}
\end{cor}


\subsection{Moduli spaces of stable pairs}

In \cite{PT09}, Pandharipande--Thomas  introduced a theory of curve-counting via stable pairs.

\begin{defn}\label{stab-pair-def}
Let $X$ be a smooth projective $3$-fold.
A {\it stable pair} on $X$ is a pair $(F,s)$ satisfying the following conditions:
\begin{itemize}
  \item[(i)] $F$ is a $1$-dimensional pure sheaf, i.e., there are no $0$-dimensional subsheaves; 
  \item[(ii)] $s: \CO_{X} \rightarrow F$ is a non-trivial morphism having $0$-dimensional cokernel.
\end{itemize}
\end{defn}

For a stable pair $(F,s)$, let $C_{F}$ be the support of $F$. 
By \cite[Lemma 1.6]{PT09}, 
the image of $s$ is the structure sheaf $\CO_{C_{F}}$ and $C_{F}$ is a Cohen-Macaulay curve.
Here, a Cohen-Macaulay curve means a curve of pure dimension $1$ and possibly non-reduced, but without embedded points.
Fix a homology class $\beta\in H_{2}(X,\ZN)$ and $n\in \ZN$, the {\it moduli space $P_{n}(X,\beta)$ of stable pairs} is defined as
\begin{equation*}
P_{n}(X,\beta):=
\left \{\begin{array}{c} 
\textrm{stable pairs } (F,s) \\
\textrm{ with }   
{ [F] =\beta \textrm{ and } \chi(F)=n}
 \end{array}\right\},
\end{equation*}   
where $[F]$ is the homology class of the support of $F$.

\begin{thm}[{\cite{LeP93,PT09}}]
The moduli space $P_{n}(X,\beta)$ is a fine projective moduli space.
\end{thm}

\begin{ex}
Some moduli spaces of stable pairs associated with low degree curves in $\PB^{3}$ are provided as follows:
\begin{enumerate}
\item[(1)] $P_{1}(\PB^{3},[\ell]) \cong  \Hilb^{t+1}(\PB^{3})$ and it is of dimension $4$ and isomorphic to the Grassmann manifold ${\rm Gr}(2,4)$;
\item[(2)] $P_{1}(\PB^{3},2[\ell])\cong  \Hilb^{2t+1}(\PB^{3})$ and it is of dimension $8$ and is isomorphic to a $\PB^{5}$-bundle over $(\PB^{3})^{\vee}$;
\item[(3)] $P_{0}(\PB^{3},3[\ell])\cong \Hilb^{3t}(\PB^{3})$ and it is of dimension $12$ and is isomorphic to a $\PB^{9}$-bundle over $(\PB^{3})^{\vee}$,
\end{enumerate}
where $[\ell]$ is the homology class of  a line $\ell\subset \PB^{3}$.
\end{ex}

In \cite[Proposition 1.21]{PT09}, a stable pair $(F,s)$ is viewed as a two-term complex $\{\CO_{X}\xrightarrow{s} F\} $ uniquely up to quasi-isomorphisms.
Recall that an object $E\in \DC(X)$ is called  a {\it PT stable object} if $E$ is quasi-isomorphic to a two-term complex $\{\CO_{X}\xrightarrow{s} F\}$ such that $(F,s)$ is a stable pair. 
If $(F,s)$ is a stable pair, then the two-term complex $\{\CO_{X}\xrightarrow{s} F\}$ is a stable pair.
Moreover, one has the following equivalent description for stable pairs. 

\begin{lem}[{\cite{PT09},\cite[Lemma 2.9]{Rez24a}}]\label{stab-pair-equ-def}
An object $E\in \DC(X)$ is a PT stable object if and only if the following conditions hold: 
$(1)$\, $\CH^{i}(E)=0$ if $i\neq 0,1$;
$(2)$\, $\CH^{0}(E)$ is the ideal sheaf of a Cohen-Macaulay curve; 
$(3)$\, $\CH^{1}(E)$ is a sheaf with $0$-dimensional support, and
$(4)$\, $\Hom(\CO_{x},E[1])=0$ for all points $x\in X$.
\end{lem}

Here, we present a proof for the reader's convenience.

\begin{proof}
Suppose that a PT stable object $E\in \DC(X)$ is isomorphic to $I^{\bullet}:=\{\CO_{X}\xrightarrow{s} F\}$, where $(F,s)$ is a stable pair.
Then, we have an exact sequence
$$
\xymatrix@C=0.5cm{
0 \ar[r] & I_{C} \ar[r] & \CO_{X} \ar[r]^{s} \ar[r] & F \ar[r] & Q \ar[r] & 0,
}
$$
where $C$ is a Cohen-Macaulay curve and $Q$ is a $0$-dimensional sheaf. 
Taking cohomology sheaves of the exact triangle
$$
\xymatrix@C=0.5cm{
I_{C} \ar[r] & I^{\bullet} \ar[r] & Q[-1] \ar[r] & I_{C}[1],
}
$$
we obtain $\CH^{0}(E)\cong \CH^{0}(I^{\bullet})\cong I_{C}$ and $\CH^{1}(E)\cong \CH^{1}(I^{\bullet})\cong Q$.
For any point $x\in X$,
applying $\Hom(\CO_{x},-)$ to the exact triangle
$
\xymatrix@C=0.4cm{
F\ar[r] & I^{\bullet}[1] \ar[r] & \CO_{X}[1] \ar[r] & F[1],
}
$
we have $\Hom(\CO_{x},E[1])\cong\Hom(\CO_{x},F)=0$ since $F$ is a pure sheaf.

Conversely, by condition (1),
we may assume that $E$ is a two-term complex. 
Then, we have an exact triangle
\begin{equation}\label{stable_pair_E_SES}
\xymatrix@C=0.5cm{
I_{C} \ar[r] & E \ar[r] & Q[-1] \ar[r] & I_{C}[1],
} 
\end{equation}
where $I_{C}:=\CH^{0}(E)$ and $Q:=\CH^{1}(E)$.
Note that $\Hom(I_{C},\CO_{X}) \cong \Hom(\CO_{X},\CO_{X})= \CN$.
Applying $\Hom(-,\CO_{X})$ to \eqref{stable_pair_E_SES}, we obtain $\Hom(E,\CO_{X})\cong \Hom(I_{C},\CO_{X})=\CN$. 
Let $f\in \Hom(E,\CO_{X})$ be a non-trivial morphism. 
Then, we have an exact triangle
\begin{equation}\label{Cone-E-OX-SES}
\xymatrix@C=0.5cm{
E \ar[r]^{f} & \CO_{X} \ar[r]^{g\,\,\,\,} & C(f) \ar[r] & E[1],
}
\end{equation}
where $C(f)$ is the cone of $f$.
Taking cohomology sheaves of \eqref{Cone-E-OX-SES}, we have $\mathcal{H}^{i}(C(f))=0$ for $i\neq 0$. 
Hence, there exists a quasi-isomorphism $q: C(f) \rightarrow \mathcal{H}^{0}(C(f))$.
We denote by $F:=\mathcal{H}^{0}(C(f))$ and by $s=q\circ g:\CO_{X} \rightarrow F$.
We set $I^{\bullet}:=\{\CO_{X} \xrightarrow{s} F\}$ the two-term complex.
By the definition of $s$, we obtain a morphism of exact triangles
\begin{equation}\label{mor-ext-tri-1}
\xymatrix@C=0.5cm{
E \ar[r]^{f} \ar[d] & \CO_{X} \ar[r]^{g\,\,\,\,} \ar[d]_{id} & C(f) \ar[r] \ar[d]_{q} & E[1] \ar[d]\\
I^{\bullet} \ar[r] & \CO_{X} \ar[r]^{s} & F \ar[r] & I^{\bullet}[1].
}
\end{equation}
It follows that $E$ is quasi-isomorphic to $ I^{\bullet}$.
For any $x\in X$,
applying $\Hom(\CO_{x},-)$ to the bottom exact triangle in \eqref{mor-ext-tri-1},
the condition (4) implies $\Hom(\CO_{x},F)=0$, i.e., $F$ is a pure sheaf.
Hence, by conditions (2) and (3), the pair $(F,s)$ is a stable pair.
As a result, by the definition, $E\in \DC(X)$ is a PT stable object.
\end{proof}

\begin{rem}\label{PT-stab-pair-equ-defn}
For a homology class $\beta\in H_{2}(X,\ZN)$ and $n\in \ZN$, by Lemma \ref{stab-pair-equ-def}, we have
\begin{equation*}
P_{n}(X,\beta)=
\left \{\begin{array}{c} 
\textrm{PT stable objects } E\in \DC(X) \textrm{ with}\\  
\ch(E)=(1,0,-\beta,-n+\frac{\beta\cdot c_{1}(X)}{2})
 \end{array}\right\}.
\end{equation*}      
\end{rem}

The following result is well-known to experts (see e.g.  \cite[Proposition 6.1.1]{Bay09}, \cite[Proposition 3.12]{Tod10} and \cite{Rez24a}). 
It plays a crucial role in the study of the moduli spaces of stable pairs via wall crossings in Bridgeland stability. 
However, we find no proof in literature.
Here, we present a proof for the reader's convenience.
Suppose now that $X$ is a Fano threefold of Picard number $1$.
For the Chern character $v=(1,0,-\beta,-n+\frac{\beta\cdot c_{1}(X)}{2})$,
we denote by $\Gamma$ the left branch of the hyperbola given by $\mu_{\alpha,\beta}(v)=0$.

\begin{prop}\label{Bri-mod=PT-mod}
Let $(\alpha,\beta)$ be above the largest wall and close to $\Gamma$ on the right side,
and $\sigma:=\lambda_{\alpha,\beta,s}$ a Birdgeland stability condition with $s\gg 1$.
Then the moduli space $P_{n}(X,\beta)$ is isomorphic to the Bridgeland moduli space $M_{\sigma}(v)$.
\end{prop}

\begin{proof} 
Assume $E\in M_{\sigma}(v)$.
Then $E[1]\in \Coh^{\alpha,\beta}(X)$ and $E[1]$ is $\lambda_{\alpha,\beta,s}$-stable.
Since $\CO_{x}$ is $\lambda_{\alpha,\beta,s}$-stable, 
so $\Hom(\CO_{x},E[1])=0$ for all $x\in X$.
We need to show that $E$ is a PT stable object.
Note that there is a short exact sequence
\begin{equation}\label{E-def-ses}
\xymatrix@C=0.5cm{
0 \ar[r] & \mathcal{H}^{0}_{\beta}(E)[1] \ar[r] & E[1] \ar[r] & \mathcal{H}^{1}_{\beta}(E) \ar[r] & 0
}
\end{equation}
in $\Coh^{\alpha,\beta}(X)$.
By Corollary  \ref{tilt-stab-Bri-stab-cor},
$\CH_{\beta}^{0}(E)$ is $\sigma_{\alpha,\beta}$-stable and $\CH_{\beta}^{1}(E)$ is a torsion sheaf $Q$ supported in dimension $0$. 
By Proposition \ref{limit-tilt-stab-Gie-stab-prop}, we obtain $\CH_{\beta}^{0}(E)$ is an ideal sheaf $I_{C}$ of a curve $C$. 
By taking cohomology sheaves of \eqref{E-def-ses}, we deduce $\CH^{0}(E)\cong I_{C}$ and $\CH^{1}(E)\cong Q$.
Applying $\Hom(\CO_{x},-)$ to \eqref{E-def-ses},
we get $\Hom(\CO_{x},I_{C}[1])=0$ for all $x\in X$.
By applying $\Hom(\CO_{x},-)$ to the structure sheaf sequence 
$$
\xymatrix@C=0.5cm{
0 \ar[r] & I_{C} \ar[r] & \CO_{X} \ar[r] & \CO_{C} \ar[r] & 0,
}
$$
we have $\Hom(\CO_{x},\CO_{C})=0$ for all $x\in X$. 
Thus, $C$ is a Cohen-Macaulay curve.
According to Lemma \ref{stab-pair-equ-def}, $E$ is a PT stable object.
Hence, Remark \ref{PT-stab-pair-equ-defn} implies $E\in P_{n}(X,\beta)$.

Conversely, let $E\in P_{n}(X,\beta)$ be a PT stable object.
By Lemma \ref{stab-pair-equ-def},
we may assume that $E:=\{\CO_{X}\xrightarrow{s}F\}$, where $(F,s)$ is a stable pair. Let $I_{C}:=\CH^{0}(E)$ be the ideal sheaf of the Cohen-Macaulay curve $C$ and $Q:=\CH^{1}(E)$ the torsion sheaf with $0$-dimensional support. 
It is sufficient to show that $E[1]\in \Coh^{\alpha,\beta}(X)$ and $E[1]$ is $\lambda_{\alpha,\beta,s}$-semistable.
For any $\beta<0,\alpha\gg 0$, by Proposition \ref{limit-tilt-stab-Gie-stab-prop}, $I_{C}$ is $\sigma_{\alpha,\beta}$-stable. Hence, we derive $I_{C}[1]\in \Coh^{\alpha,\beta}(X)$. 
Note that $Q\in \Coh^{\alpha,\beta}(X)$.
Since $I_{C}[1]$, $Q$ and $E[1]$ fit into an exact triangle as \eqref{stable_pair_E_SES}, we have $E[1]\in \Coh^{\alpha,\beta}(X)$.
By Corollary \ref{tilt-stab-Bri-stab-cor}, 
it suffices to prove that $E$ is $\lambda_{\alpha,\beta,s}$-semistable for $s\gg 1$.
Assume that $E$ is not $\lambda_{\alpha,\beta,s}$-semistable for $s\gg 1$. Let $A\subset E[1]$ be the $\lambda_{\alpha,\beta,s}$-semistable subobject for $s\gg 1$ in the Harder-Narasimhan filtration of $E[1]$.
From this, we see that $\lambda_{\alpha,\beta,s}(A)>\lambda_{\alpha,\beta,s}(E[1])$ for $s\gg 1$.
By Lemma \ref{tilt-stab-Bri-stab}, we know that $\CH^{-1}_{\beta}(A)$ is $\sigma_{\alpha,\beta}$-semistable and $\CH^{0}_{\beta}(A)$ is supported in dimension $0$.
We have two cases.

(i) If $\CH^{-1}_{\beta}(A)=0$, then $A\cong \CH^{0}_{\beta}(A)$ is supported in dimension $0$.
By the hypothesis $\Hom(A,E[1])\neq 0$, this contradicts Lemma \ref{stab-pair-equ-def} (4).

(ii) If $\CH^{-1}_{\beta}(A)\neq 0$, then we get $\ch_{\leq 2}(A)=\ch_{\leq 2}(\CH^{-1}_{\beta}(A)[1])$.
Since $\lambda_{\alpha,\beta,s}(A)>\lambda_{\alpha,\beta,s}(E[1])$ for $s\gg 1$, so we have 
\begin{equation}\label{tilt_tilt_stab_A}
-\frac{\ch_{1}^{\beta}(A)}{\ch_{2}^{\beta}(A)-\frac{\alpha^{2}}{2}\ch_{0}^{\beta}(A)}
>
-\frac{\ch_{1}^{\beta}(E[1])}{\ch_{2}^{\beta}(E[1])-\frac{\alpha^{2}}{2}\ch_{0}^{\beta}(E[1])}.
\end{equation}
In the inequality \eqref{tilt_tilt_stab_A}, 
the denominators are positive.
Note that $\ch_{1}^{\beta}(\CH^{-1}_{\beta}(A))>0$ and $\ch_{\leq 2}(E[1])=\ch_{\leq 2}(I_{C}[1])$.
Then, the inequality \eqref{tilt_tilt_stab_A} yields  
\begin{equation}\label{tilt_stab_H^-1A}
\mu_{\alpha,\beta}(\CH^{-1}_{\beta}(A))
>
\mu_{\alpha,\beta}(I_{C}).
\end{equation}
Since $A$ is a subobject of $E[1]$
and $I_{C}=\CH^{0}_{\beta}(E)$, 
taking $\CH^{-1}_{\beta}$ cohomology objects,
we have
$
\CH^{-1}_{\beta}(A) \subset I_{C}
$
in $\Coh^{\beta}(X)$.
This contradicts \eqref{tilt_stab_H^-1A}, 
since $I_{C}$ is $\sigma_{\alpha,\beta}$-stable.
\end{proof}


\section{Wall and chamber structures}\label{Wall-chamber}
 
It is well-known that there exists a locally finite wall and chamber structure in the stability manifold.
Let $v=(1,0,-5,11)$ be the Chern character of the ideal sheaf of a quintic genus $2$ space curve.
In this section, we describe the walls in tilt stability and Bridgeland stability for the class 
$v$ and present a geometric description of these walls. 
 
Firstly, we compute all possible walls in tilt stability with respect to the class $v$. Let $\Gamma$ be the left branch of the hyperbola
$\beta^{2}-\alpha^{2}=10$ given by $\mu_{\alpha,\beta}(v)=0$. 
By Theorem \ref{tilt-wall-struct-thm}, these walls intersect $\Gamma$ at their apexes.
We denote the semicircle 
$$
W(c,r):=\{(\alpha,\beta)\in \RN_{>0}\times \RN | \alpha^{2}+(\beta-c)^{2}=r^{2}\}.
$$

\begin{prop}\label{main-numerical-walls}
There exist three possible walls intersecting $\Gamma$ for tilt semistable objects of class $v$.
The walls, with either the subobject or the quotient $A$, are given as follows:
\begin{enumerate}
\item[(1)]  $W(-\frac{7}{2},\frac{3}{2})$, $\ch_{\leq 2}(A)=(1,-2,2)$ or $\ch_{\leq 2}(A)=(1,-1,-\frac{3}{2})$;
\item[(2)]  $W(-\frac{9}{2},\frac{\sqrt{41}}{2})$, $\ch_{\leq 2}(A)=(1,-1,-\frac{1}{2})$,
\item[(3)] $W(-\frac{11}{2},\frac{9}{2})$, $\ch_{\leq 2}(A)=(1,-1,\frac{1}{2})$.
\end{enumerate}
\end{prop}

\begin{proof}
Let $E$ be a tilt semistable object with $\ch(E)=v=(1,0,-5,11)$.
Suppose that $E$ is $\sigma_{\alpha,\beta}$-semistable for some $\alpha$ and $\beta$. 
Then, by Theorem \ref{BMT-inequality}, we have the BMT inequality
$$
Q_{\alpha,\beta}(E)=5(\alpha^{2}+\beta^{2})+33\beta+50\geq 0.
$$
Note that $Q_{\alpha,\beta}(E)$ defines a semicircle intersecting with $\beta=-3$. 
As a result, it is sufficient to determine all possible walls along $\beta=-3$.
Assume that there is such an actual wall induced by a short exact sequence
\begin{equation}\label{wall-defn-exact-sq}
\xymatrix@C=0.5cm{
0 \ar[r] & A \ar[r] & E \ar[r] & B \ar[r] & 0
}
\end{equation}
of $\sigma_{\alpha,-3}$-semistable objects in $\Coh^{-3}(\PB^{3})$. 
We write the twisted Chern character $\ch_{\leq 2}^{-3}(A)=(a,b,\frac{c}{2})$, where $a,b,c\in \ZN$. 
Then, we have $\ch_{\leq 2}(A)=(a,(b-3a),\frac{c+9a-6b}{2})$. By changing the roles of $A$ and $B$, we may assume $a\geq 1$. 
Note that $\ch_{\leq 2}^{-3}(E)=(1,3,-\frac{1}{2})$. 
Under the assumption \eqref{wall-defn-exact-sq},
by Proposition \ref{tilt-Bog-ineq} and Definition \ref{wall-def}, we have
$0\leq \Delta_{H}(A)\leq \Delta_{H}(E)$
and $\mu_{\alpha,-3}(A)=\mu_{\alpha,-3}(E)$.
Then, we obtain 
\begin{eqnarray}
b^{2}-10 &\leq&  ac \leq b^{2},  \label{delta-inequ} \\
(3a-b)\alpha^{2}&=&  3c+b.   \label{mu-slope-equ}
\end{eqnarray}
It follows from $A\in \Coh^{-3}(\PB^{3})$ that $0< \ch_{1}^{-3}(A)< \ch_{1}^{-3}(E)$. 
Then, we get $b\in\{1,2\}$.
In the following, the proof discusses two cases:

{\bf Case}  $b=1$. 
Since $a\geq 1$, by \eqref{delta-inequ} and \eqref{mu-slope-equ},
we obtain $c\geq 0$ and $0\leq ac\leq 1$.
Clearly, we have $c=0$ or $c=1$.
(i) If $c=0$, then $\ch_{\leq 2}(A)=(a,1-3a,\frac{9a-6}{2})$. 
Hence, we have $c_{2}(A)=\frac{3a(3a-5)+7}{2} \notin \ZN $, a contradiction.
(ii) If $c=1$, then $a=1$. 
Thus, we derive $\ch_{\leq 2}(A)=(1,-2,2)$. 
Then, we get the possible wall $W(E,A)=W(-\frac{7}{2},\frac{3}{2})$.

{\bf Case}  $b=2$.
Since $a\geq 1$, by \eqref{delta-inequ} and \eqref{mu-slope-equ}, 
we get $c\geq 0$ and $0\leq ac\leq 4$.
Since the second Chern class $c_{2}(A)=\frac{6b-6ab+9a(a-1)+b^{2}-c}{2}\in \ZN$, so $c$ must be even. 
Hence, we obtain $c\in \{0,2,4\}$.
{\bf (i)} If $c=0$, then $\ch_{\leq 2}(A)=(a,2-3a,\frac{9a-12}{2})$. 
Next, we consider three cases for $a$.
(1) If $a=1$, then $\ch_{\leq 2}(A)=(1,-1,-\frac{3}{2})$.
Thus, we have the possible wall $W(E,A)=W(-\frac{7}{2},\frac{3}{2})$.
(2) If $a=2$, then $\ch_{\leq 2}(A)=(2,-4,3)$ and $\ch_{\leq 2}(B)=(-1,4,-8)$. 
Since $B$ is $\sigma_{\alpha,\beta}$-semistable, by \cite[Proposition 4.6]{Sch20a}, we deduce $B\cong \CO(-4)[1]$. Hence, we have $\ch(A)=(2,-4,3,\frac{1}{3})$. 
Note that $A$ is $\sigma_{\alpha,\beta}$-semistable for $(\alpha,\beta)\in W(E,A)=W(-\frac{13}{4},\frac{3}{4})$.
Thanks to Theorem \ref{BMT-inequality}, 
we obtain the BMT inequality
$
Q_{\alpha,\beta}(A)=2(\alpha^{2}+\beta^{2})+14\beta+22 \geq 0,
$
a contradiction, since the semicircle defined by $Q_{\alpha,\beta}(A)$ is above $W(E,A)$.
(3) If $a\geq 3$, then $\ch_{\leq 2}(B)=(1-a,3a-2,\frac{2-9a}{2})$. Thus, we have $\frac{\ch_{2}(B)}{\ch_{0}(B)}>3|\mu_{H}(B)|-\frac{9}{2}$, a contradiction with \cite[Proposition 3.2]{Li19}. 
{\bf (ii)} If $c=2$, then $a\in \{1,2\}$. 
(1) If $a=1$, then $\ch_{\leq 2}(A)=(1,-1,-\frac{1}{2})$. 
Thus, we get the possible wall $W(E,A)=W(-\frac{9}{2},\frac{\sqrt{41}}{2})$.
(2) If $a=2$, then $\ch_{\leq 2}(A)=(2,-4,4)$. 
Hence, we get $\ch_{\leq 2}(B)=(-1,4,-9)$, 
a contradiction with \cite[Proposition 3.2]{Li19}.
{\bf (iii)} If $c=4$, then $a=1$ and $\ch_{\leq 2}(A)=(1,-1,\frac{1}{2})$.
Therefore, we have the possible wall $W(E,A)=W(-\frac{11}{2},\frac{9}{2})$. 
\end{proof}

\begin{rem}
Generally speaking, for a smooth, irreducible, and non-degenerate curve $C\subset \PB^{3}$ of degree $d\geq 3$, a famous theorem of Castelnuovo states that its genus $g$ satisfies: (i) if $d$ is even, then $g\leq d^{2}/4-d+1$; 
(ii) if $d$ is odd, then $g\leq  (d^{2}-1)/4-d+1$.
Let $C$ be of maximal genus.
\begin{enumerate}
\item[(i)] If $d=2m$, then $\ch(I_{C})=(1,0,-2m,m^{2}+2m)$. 
The BMT inequality $Q_{\alpha,\beta}(I_{C})\geq 0$ gives a semicircle 
$\alpha^{2}+[\beta+\frac{3}{4}(m+2)]^{2}=\frac{9m^2-28m+36}{16}$. 
Thus, it suffices to determine the numerical walls along $\beta=-m$.
\item[(ii)] If $d=2m+1$, then $\ch(I_{C})=(1,0,-2m-1,m^{2}+3m+1)$. 
The BMT inequality $Q_{\alpha,\beta}(I_{C})\geq 0$ gives a semicircle $
\alpha^{2}+[\beta+\frac{3(m^2+3m+1)}{2(2m+1)}]^2=\frac{9m^{4}-10m^{3}+3m^{2}+6m+1}{4(2m+1)^{2}}$. Hence, it suffices to determine the numerical walls along $\beta=-m-1$.
\end{enumerate}
\end{rem}
 
Next, we will describe the destabilizing pairs of tilt semistable objects whose extensions completely determine all strictly semistable objects at all the walls in Proposition \ref{main-numerical-walls}.

\begin{thm}\label{main-actual-wall}
There are three walls for tilt semistable objects of class $v$ in the $(\alpha,\beta)$-plane with $\beta<0$.
Let $B$ be a Gieseker semistable sheaf fitting into the short exact sequence 
\begin{equation}\label{B_SES_S4}
\xymatrix@C=0.5cm{
0 \ar[r] & \CO(-3)^{\oplus 2} \ar[r] & B \ar[r] & \CO(-4)^{\oplus 2}[1] \ar[r] & 0,
}
\end{equation}
$Z_{i}$ a $0$-dimensional subscheme of length $i$ in a plane $V$, $\ell_{i}$ a line plus $i$ embedding/floating points, $Z_{i}^{\prime}$ a $0$-dimensional subscheme of length $i$ and $C_{2}$  a conic. 
\begin{center}
\renewcommand\arraystretch{1.4}
\begin{tabular}{p{1.4in}p{2.28in}}
    \hline
     $(\beta+\frac{7}{2})^{2}+\alpha^{2}=(\frac{3}{2})^{2}$& $\langle \CO(-2),B\rangle$, $\langle I_{C_{2}}(-1),\CO_{V}(-3) \rangle$ \\
    \hline
    $(\beta+\frac{9}{2})^{2}+\alpha^{2}=(\frac{\sqrt{41}}{2})^{2}$ &
    $\langle I_{\ell_{i}}(-1), I_{Z_{i}/V}(-4)\rangle, i\in \{0,1\} $ \\
    \hline
    $(\beta+\frac{11}{2})^{2}+\alpha^{2}=(\frac{9}{2})^{2}$ & $\langle I_{Z_{i}^{\prime}}(-1), I_{Z_{i}/V}(-5) \rangle, i\in\{0,1,2,3,4\}$ \\
    \hline
  \end{tabular}
\end{center}
\end{thm}

\begin{proof}
In Proposition \ref{main-numerical-walls}, assume the third Chern character $\ch_{3}(A):=e\in \frac{1}{6}\ZN$.
The proof will discuss the three possible walls case-by-case.

{\bf Case (1).} For the wall $W(-\frac{7}{2},\frac{3}{2}):(\beta+\frac{7}{2})^{2}+\alpha^{2}=(\frac{3}{2})^{2}$, 
by Proposition \ref{main-numerical-walls},
we get $\ch_{\leq 2}(A)=(1,-2,2)$ or  $\ch_{\leq 2}(A)=(1,-1,-\frac{3}{2})$.
Hence, we have two cases:
\begin{enumerate}
\item[(i)] If $\ch(A)=(1,-2,2,e)$, 
then $\ch^{-3}(A)=(1,1,\frac{1}{2},e+\frac{3}{2})$ and $\ch(B)=(0,2,-7,11-e)$.
By \cite[Lemma  5.4 (1)]{Sch20a}, 
we obtain $A\cong I_{Z^{\prime}}(-2)$, where $Z^{\prime}$ is a $0$-dimensional subscheme of length $-e-\frac{4}{3}\geq 0$.
Since $B$ is $\sigma_{\alpha,\beta}$-semistable along the wall $W(-\frac{7}{2},\frac{3}{2})$, applying Theorem \ref{BMT-inequality} to $B$, 
we get $e\geq -2$ and thus $e=-\frac{4}{3}$. 
It follows that  $A\cong \CO(-2)$ and $\ch(B)=(0,2,-7,\frac{37}{3})$.
By \cite[Theorem 4.4 (ii)]{Sch23}, $B$ is a Gieseker semistable sheaf fitting into a short exact sequence
$
\xymatrix@C=0.4cm{
0 \ar[r] & \CO(-3)^{\oplus 2} \ar[r] & B \ar[r] & \CO(-4)^{\oplus 2}[1] \ar[r] & 0.
}
$
\item[(ii)] If $\ch(A)=(1,-1,-\frac{3}{2},e)$, then $\ch(A(1))=(1,0,-2,e-\frac{11}{6})$.
Since $A(1)$ is tilt semistable,
by \cite[Proposition 3.2]{MS20}, we have $e\leq \frac{29}{6}$.
Since $\ch^{-3}(B)=(0,1,-\frac{1}{2},5-e)$ and $B$ is tilt semistable,
by \cite[Lemma 5.4 (2)]{Sch20a}, 
we obtain $B\cong I_{Z/V}(-3)$, 
where $Z$ is a $0$-dimensional subscheme of length $e-\frac{29}{6}\geq 0$ and $V\subset \PB^{3}$ is a plane.
Clearly, $e=\frac{29}{6}$. 
Thus, we have $B\cong \CO_{V}(-3)$ and $\ch(A(1))=(1,0,-2,3)$.
By \cite[Theorem 4.1]{Sch23}, 
$A(1)$ fits into a short exact sequence
$
\xymatrix@C=0.4cm{
0 \ar[r] & \CO(-1) \ar[r] & A(1) \ar[r] & \CO_{V^{\prime}}(-2) \ar[r] & 0,
}
$
where $V^{\prime}\subset \PB^{3}$ is a plane. 
Hence, we get $A\cong I_{C_{2}}(-1)$, where $C_{2}\subset V^{\prime}$ is a conic.
\end{enumerate}

{\bf Case (2).}
For the wall $W(-\frac{9}{2},\frac{\sqrt{41}}{2}):(\beta+\frac{9}{2})^{2}+\alpha^{2}=(\frac{\sqrt{41}}{2})^{2}$,
we have $\ch(A)=(1,-1,-\frac{1}{2},e)$ and $\ch(B)=(0,1,-\frac{9}{2},11-e)$.
Since the vertical line $\beta=-2$ intersects $W(-\frac{9}{2},\frac{\sqrt{41}}{2})$ at the point $(\frac{\sqrt{31}}{2},-2)$, 
so $A$ and $B$ are $\sigma_{\frac{\sqrt{31}}{2},-2}$-semistable.
Note that $\ch^{-2}(A)=(1,1,-\frac{1}{2},e-\frac{5}{3})$ and $\ch^{-2}(B)=(0,1,-\frac{5}{2},4-e)$.
Therefore, both $A$ and $B$ have no walls along $\beta=-2$.
By \cite[Lemma  5.4 (1)]{Sch20a}, 
we obtain $A\cong I_{Z^{\prime}}(-1)$, 
where $Z^{\prime}$ is a $0$-dimensional subscheme of length $\frac{11}{6}-e\geq 0$.
Using \cite[Lemma  5.4 (2)]{Sch20a}, we deduce $B\cong I_{Z/V}(-4)$, 
where $Z$ is a $0$-dimensional subscheme of length $e-\frac{5}{6}\geq 0$ and $V\subset \PB^{3}$ is a plane.
It follows that $e=\frac{5}{6}$ or $e=\frac{11}{6}$.
If $e=\frac{5}{6}$, then $A\cong I_{\ell_{1}}(-1)$ and $B\cong \CO_{V}(-4)$, 
where $\ell_{1}$ is a line plus one point, and $V$ is a plane.
If $e=\frac{11}{6}$, then $A\cong I_{\ell_{0}}(-1)$ and $B\cong I_{Z_{1}/V}(-4)$, where $\ell_{0}$ is a line, $V$ is a plane and $Z_{1}$ is a point in $V$.

{\bf Case (3).} 
For the wall $W(-\frac{11}{2},\frac{9}{2}):(\beta+\frac{11}{2})^{2}+\alpha^{2}=(\frac{9}{2})^{2}$, 
we have $\ch(A)=(1,-1,\frac{1}{2},e)$ and $\ch(B)=(0,1,-\frac{11}{2},11-e)$.
Since the vertical line $\beta=-2$ intersects with $W(-\frac{11}{2},\frac{9}{2})$ at the point $(2\sqrt{2},-2)$, hence $A$ and $B$ are $\sigma_{2\sqrt{2},-2}$-semistable.
Note that $\ch^{-2}(A)=(1,1,\frac{1}{2},e+\frac{1}{3})$ and $\ch^{-2}(B)=(0,1,-\frac{7}{2},2-e)$.
Thus, both $A$ and $B$ have no walls along $\beta=-2$.
According to \cite[Lemma  5.4 (1)]{Sch20a}, 
we have $A\cong I_{Z^{\prime}}(-1)$, where $Z^{\prime}$ is a $0$-dimensional subscheme of length $-e-\frac{1}{6}\geq 0$.
By \cite[Lemma  5.4 (2)]{Sch20a}, we obtain $B\cong I_{Z/V}(-5)$, where $Z$ is a $0$-dimensional subscheme of length $e+\frac{25}{6}\geq 0$ and $V\subset \PB^{3}$ is a plane.
As a result, 
we get $e=-\frac{1}{6}-i$ for $0\leq i\leq 4$.
Therefore, we have $ A\cong I_{Z_{i}^{\prime}}(-1)$ and  $B\cong I_{Z_{4-i}/V}(-5)$, 
where $Z_{i}^{\prime}$ is a $0$-dimensional subscheme of length $i$ and $Z_{4-i}$ is a $0$-dimensional subscheme of length $4-i$ in $V$.
\end{proof}

Now, we will prove the main result of this section.

\begin{thm}\label{main-wall-in-Bridgeland-stability}
The walls in Bridgeland stability on both sides of $\Gamma$ (sufficiently close to $\Gamma$) are given as follows:
\begin{center}
\renewcommand\arraystretch{1.4}
\begin{tabular}{p{1.44in}p{1.28in}p{1.68in}}
    \hline
    Walls in tilt stability & Walls on the left side & Walls on the right side \\
    \hline
     $(\beta+\frac{7}{2})^{2}+\alpha^{2}=(\frac{3}{2})^{2}$& $\langle \CO(-2),B\rangle$ & $\langle I_{C_{2}}(-1),\CO_{V}(-3) \rangle$\\
    & $\langle I_{C_{2}}(-1),\CO_{V}(-3) \rangle$ & $\langle \CO(-2),B\rangle$ \\
    \hline
    $(\beta+\frac{9}{2})^{2}+\alpha^{2}=(\frac{\sqrt{41}}{2})^{2}$ &
    $\langle I_{\ell_{1}}(-1), \CO_{V}(-4)\rangle $ & $\langle I_{\ell_{0}}(-1),i_{V_\ast}(I_{Z_{1}/V}^{\vee}(-4))\rangle$\\
    & $\langle I_{\ell_{0}}(-1), I_{Z_{1}/V}(-4)\rangle $ & $\langle \{\CO(-1)\stackrel{s}{\rightarrow }\CO_{\ell_{0}}\},\CO_{V}(-4)\rangle$ \\
    \hline
    & $\langle I_{Z_{4}^{\prime}}(-1), \CO_{V}(-5) \rangle$ & \\
    & $\langle I_{Z_{3}^{\prime}}(-1), I_{Z_{1}/V}(-5) \rangle$ & \\
    $(\beta+\frac{11}{2})^{2}+\alpha^{2}=(\frac{9}{2})^{2}$ & $\langle I_{Z_{2}^{\prime}}(-1), I_{Z_{2}/V}(-5) \rangle$ & $\langle \CO(-1),i_{V_\ast}(I_{Z_{4}/V}^{\vee}(-5)) \rangle$ \\
    & $\langle I_{Z_{1}^{\prime}}(-1), I_{Z_{3}/V}(-5) \rangle$ & \\
    & $\langle \CO(-1), I_{Z_{4}/V}(-5) \rangle$ & \\
    \hline
\end{tabular}
\end{center}
where $B$ is a Gieseker semistable sheaf fitting into the exact sequence \eqref{B_SES_S4},
$Z_{i}$ is a $0$-dimensional subscheme of length $i$ in a plane $V$, $\ell_{i}$ is a line plus $i$ embedding/floating points, $Z_{i}^{\prime}$ is a $0$-dimensional subscheme of length $i$ and $C_{2}$ is a conic and the zero locus of $s$ is a point in the line $\ell_{0}$.
\end{thm}

\begin{proof}
According to Lemma \ref{tilt-stab-Bri-stab} and Theorem \ref{main-actual-wall}, we obtain the walls on the left side of $\Gamma$.
In the following, we will describe the walls on the right side of $\Gamma$.
By Lemma \ref{tilt-stab-Bri-stab},
the corresponding walls in Bridgeland stability are given by the destabilizing pairs $\langle A,B \rangle$ as in Proposition \ref{main-numerical-walls}.
We assume the third Chern character $\ch_{3}(A):=e\in \frac{1}{6}\ZN$. The proof will discuss the walls case-by-case.

{\bf Case (1).} For the wall $(\beta+\frac{7}{2})^{2}+\alpha^{2}=(\frac{3}{2})^{2}$, 
by Proposition \ref{main-numerical-walls},
we have $\ch_{\leq 2}(A)=(1,-2,2)$ or $\ch_{\leq 2}(A)=(1,-1,-\frac{3}{2})$.
Thus, we have two cases.
(i) If $\ch(A)=(1,-2,2,e)$, then
 $\ch(A(1))=(1,-1,\frac{1}{2},e+\frac{5}{6})$.
According to \cite[Lemma 3.9]{Rez24a}, we have $A(1)\cong \CO(-1)$.
Hence, we obtain $A\cong \CO(-2)$ and $\ch(B)=(0,2,-7,\frac{37}{3})$. 
Since $B$ is $\lambda_{\alpha,\beta,s}$-semistable, by Lemma \ref{tilt-stab-Bri-stab}, we obtain $\CH^{0}_{\beta}(B)$ is $\sigma_{\alpha,\beta}$-semistable and $\CH^{1}_{\beta}(B)$ is supported in dimension $0$.
Set $\ch(\CH^{1}_{\beta}(B)):=(0,0,0,t)$, where integer $t \geq 0$.
Then, we get $\ch(\CH^{0}_{\beta}(B))=(0,2,-7,\frac{37}{3}+t)$. 
By \cite[Theorem 4.4 (ii)]{Sch23}, we conclude that $t=0$ and $B\cong \CH^{0}_{\beta}(B)$ is a Gieseker semistable sheaf fitting into a short exact sequence
\eqref{B_SES_S4}.
(ii) If $\ch(A)=(1,-1,-\frac{3}{2},e)$, by \cite[Lemma 3.12]{Rez24b}, then we get $A\cong (\CO(-1)\stackrel{s}{\rightarrow}\CO_{C_{2}}(l-1))$, where $C_{2}$ is a conic and the zero locus of $s$ is a $0$-dimensional subscheme of length $l:=\frac{29}{6}-e\geq 0$.
Since $\ch(B)=(0,1,-\frac{7}{2},11-e)$, applying \cite[Lemma 3.11]{Rez24b} to $B$, we derive $B\cong i_{V_{\ast}}I_{Z/V}^{\vee}(-3)$, where $Z$ is a $0$-dimensional subscheme of length $e-\frac{29}{6}\geq 0$ and $I_{Z/V}^{\vee}$ is the derived dual of $I_{Z/V}$.
Therefore, we deduce $e=\frac{29}{6}$. 
Consequently, we have $A\cong I_{C_{2}}(-1)$ and $B\cong \CO_{V}(-3)$.

{\bf Case (2).} For the wall $(\beta+\frac{9}{2})^{2}+\alpha^{2}=(\frac{\sqrt{41}}{2})^{2}$, we have
$\ch(A)=(1,-1,-\frac{1}{2},e)$ and $\ch(B)=(0,1,-\frac{9}{2},11-e)$.
By \cite[Lemma 3.12]{Rez24b}, we obtain $A\cong (\CO(-1) \xrightarrow{s} \CO_{\ell}(l^{\prime}-1))$, 
where $\ell$ is a line and the zero locus of $s$ is a $0$-dimensional subscheme of length $l^{\prime}:=\frac{11}{6}-e\geq 0$. 
Since $\ch(B)=(0,1,-\frac{7}{2},11-e)$, 
using \cite[Lemma 3.11]{Rez24b}, we have $B\cong i_{V_\ast}I_{Z/V}^{\vee}(-4)$, where $Z$ is a $0$-dimensional subscheme of length $e-\frac{5}{6}\geq 0$.
Therefore, we obtain $e\in \{\frac{5}{6},\frac{11}{6}\}$.
(i) If $e=\frac{5}{6}$, then $B\cong \CO_{V}(-4)$ and $A\cong (\CO(-1) \xrightarrow{s} \CO_{\ell})$, 
where the zero locus of $s$ is one point. 
(ii) If $e=\frac{11}{6}$, then $A\cong (\CO(-1) \xrightarrow{s} \CO_{\ell}(-1))\cong I_{\ell}(-1)$ and $B\cong i_{V_\ast}I_{Z_{1}/V}^{\vee}(-4)$, 
where $Z_{1}$ is a $0$-dimensional subscheme of length $1$.
 
{\bf Case (3).} For the wall $(\beta+\frac{11}{2})^{2}+\alpha^{2}=(\frac{9}{2})^{2}$, 
we have $\ch(A)=(1,-1,\frac{1}{2},e)$ and $\ch(B)=(0,1,-\frac{11}{2},11-e)$.
 By \cite[Lemma 3.9]{Rez24b}, we have $e=-\frac{1}{6}$ and $A\cong \CO(-1)$. 
Thus, we get $\ch(B)=(0,1,-\frac{11}{2}, \frac{67}{6})$.
Then by \cite[Lemma 3.11]{Rez24a}, we obtain $B\cong i_{V_\ast}I_{Z_{4}/V}^{\vee}(-5)$, where $Z_{4}$ is a $0$-dimensional subscheme of length $4$.
\end{proof}


\section{Ext groups on walls}\label{Ext-on-walls}

In Theorem \ref{main-wall-in-Bridgeland-stability}, we described all the walls in Bridgeland stability, which are sufficiently close to the left branch $\Gamma$ of the hyperbola $\beta^{2}-\alpha^{2}=10$.
To study the behaviour of Bridgeland moduli spaces via wall-crossings on the right side of $\Gamma$, this section will calculate all relevant Ext groups related to these walls.

The first wall sufficiently close to $\Gamma$ from the right side is given by two different destabilizing pairs $\langle I_{C_{2}}(-1), \CO_{V}(-3) \rangle$ and $\langle \CO(-2),B\rangle$, where $B$ is a Gieseker semistable sheaf fitting into a short exact sequence
\begin{equation}\label{B_SES}
\xymatrix@C=0.5cm{
0 \ar[r] & \CO(-3)^{\oplus 2} \ar[r] & B \ar[r] & \CO(-4)^{\oplus 2}[1] \ar[r] & 0.
}
\end{equation}

\begin{prop}\label{M1_2_wall}
For the destabilizing pair $\langle \CO(-2),B\rangle$, 
we have
$$
\Ext^{1}(B,\CO(-2))=\CN^{12} \textrm{ and }\,
\Ext^{1}(\CO(-2),B)=0.
$$
If moreover, $B$ is Gieseker stable, then $\Ext^{1}(B,B)=\CN^{9}$.
\end{prop}

\begin{proof}
Applying $\Hom(\CO(-2),-)$ and $\Hom(-,\CO(-2))$ to \eqref{B_SES}, respectively, it follows that the first statement holds.
For the second, 
applying  $\Hom(-,\CO(-3)^{\oplus 2})$ and $\Hom(-,\CO(-4)^{\oplus 2}[1])$ to \eqref{B_SES},  
we obtain 
$$
\Ext^{1}(B,\CO(-3)^{\oplus 2})\cong \CN^{12} \textrm{, \;} \Hom(B,\CO(-4)^{\oplus 2}[1])\cong \CN^{4},
$$
and the other Ext groups are trivial.
If $B$ is Gieseker stable, 
then $\Hom(B,B)= \CN$. 
Then, applying $\Hom(B,-)$ to \eqref{B_SES}, we get $\Ext^{1}(B,B)=\CN^{9}$.
\end{proof}

\begin{prop}\label{M1_1_wall}
For the destabilizing pair $\langle I_{C_{2}}(-1), \CO_{V}(-3) \rangle$,
we have:
\begin{enumerate}
\item[(1)] $\Ext^{1}(I_{C_{2}}(-1),I_{C_{2}}(-1))=\CN^{8}$;
\item[(2)] $\Ext^{1}(\CO_{V}(-3),\CO_{V}(-3))=\CN^{3}$;
\item[(3)] $\Ext^{1}(\CO_{V}(-3),I_{C_{2}}(-1))=\CN^{8}$; and
\item[(4)] $
\Ext^{1}(I_{C_{2}}(-1),\CO_{V}(-3))=
\begin{cases}
  \CN^{3}    & \text{if $C_{2}\subset V$}, \\
  \CN^{2}    & \text{if $C_{2} \not \subset V$}.
\end{cases}
$
\end{enumerate}
\end{prop}

\begin{proof}
The first two Ext groups are clear.
For the last two Ext groups, assume that the conic $C_{2}$ lies in the hyperplane $V^{\prime}$. 
Then, we have a short exact sequence
\begin{equation}\label{CsubP}
\xymatrix@C=0.5cm{
0 \ar[r] & \CO(-1) \ar[r] & I_{C_{2}} \ar[r] & \CO_{V^{\prime}}(-2) \ar[r] & 0.
} 
\end{equation}
Because of $\Ext^{1}(\CO(-2),\CO(-1))=0$ and $\Ext^{1}(\CO(-2),\CO_{V^{\prime}}(-2))=0$,
applying $\Hom(\CO(-2),-)$ to \eqref{CsubP}, we obtain
$\Hom(\CO(-2),I_{C_{2}})=\CN^{5}$ and $\Ext^{1}(\CO(-2),I_{C_{2}})=0$.
Similarly, applying $\Hom(\CO(-3),-)$ to \eqref{CsubP}, we get
$
\Hom(\CO(-3),I_{C_{2}})= \CN^{13}.
$
Thanks to $\Hom(\CO_{V}(-3),I_{C_{2}}(-1))=0$ and $\Ext^{1}(\CO(-2),I_{C_{2}})=0$, by applying $\Hom(-,I_{C_{2}}(-1))$ to the short exact sequence
$$
\xymatrix@C=0.5cm{
0 \ar[r] & \CO(-4) \ar[r] & \CO(-3) \ar[r] & \CO_{V}(-3) \ar[r] & 0,
}
$$
we conclude $ \Ext^{1}(\CO_{V}(-3),I_{C_{2}}(-1))=\CN^{8}$.

For the last Ext group, since $\Ext^{i}(\CO(-1),\CO_{V}(-2))=0$ for $i=0,1$, 
applying $\Hom(-,\CO_{V}(-2))$ to  \eqref{CsubP}, we obtain
\begin{equation}\label{1st-wall-Ext-group-equ-1}
\Ext^{1}(I_{C_{2}}(-1),\CO_{V}(-3))\cong \Ext^{1}(\CO_{V^{\prime}},\CO_{V}).
\end{equation}
We consider two cases:
\begin{enumerate}
    \item[(i)] If $V=V^{\prime}$, then $\Ext^{1}(\CO_{V^{\prime}},\CO_{V})=\CN^{3}$;
    \item[(ii)] If $V\neq V^{\prime}$, then $\Hom(\CO_{V^{\prime}},\CO_{V})=0$. 
Applying $\Hom(-,\CO_{V})$ to the exact sequence
$
\xymatrix@C=0.4cm{
0 \ar[r] & \CO(-1) \ar[r] & \CO \ar[r] & \CO_{V^{\prime}} \ar[r] & 0,
}
$
we deduce $\Ext^{1}(\CO_{V^{\prime}},\CO_{V})=\CN^{2}$.
\end{enumerate}
As a result, the last Ext group follows from \eqref{1st-wall-Ext-group-equ-1}.
\end{proof}

The second wall sufficiently close to $\Gamma$ from the right side is given by the destabilizing pair $\langle I_{\ell_{0}}(-1),i_{V_{\ast}} I_{Z_{1}/V}^{\vee}(-4) \rangle$.

\begin{prop}\label{M2_wall}
For the destabilizing pair $\langle I_{\ell_{0}}(-1),i_{V_{\ast}} I_{Z_{1}/V}^{\vee}(-4) \rangle$, we have:
\begin{enumerate}
\item[(1)] $\Ext^{1}(I_{\ell_{0}}(-1),I_{\ell_{0}}(-1))=\CN^{4}$;
\item[(2)] $\Ext^{1}(i_{V_{\ast}} I_{Z_{1}/V}^{\vee}(-4), i_{V_{\ast}} I_{Z_{1}/V}^{\vee}(-4) )=\CN^{5}$;
\item[(3)] $\Ext^{1}(i_{V_{\ast}}I_{Z_{1}/V}^{\vee}(-4),I_{\ell_{0}}(-1)) = \CN^{13}$; and
\item[(4)] $\Ext^{1}(I_{\ell_{0}}(-1), i_{V_{\ast}}I_{Z_{1}/V}^{\vee}(-4))= \begin{cases}
\CN & \textrm{ if } Z_{1}\not \in \ell_{0,}\\
\CN^{2} & \textrm{ if } Z_{1} \in\ell_{0}.
\end{cases}$
\end{enumerate}
\end{prop}

\begin{proof}
The first two Ext groups are clear.
For the third statement, by Grothendieck--Verdier duality (see e.g. \cite[Theorem 3.34]{Huy06}), we have
\begin{eqnarray}\label{M2_wall_equ1}
\Ext^{1}(i_{V_{\ast}}I_{Z_{1}/V}^{\vee}(-4),I_{\ell_{0}}(-1))
&\cong & 
\Ext^{1}(I_{Z_{1}/V}^{\vee}(-4),i_{V}^{\ast}I_{\ell_{0}}[-1])\nonumber\\
&\cong&
H^{0}(i_{V}^{\ast}I_{\ell_{0}}\otimes I_{Z_{1}/V}(4)).
\end{eqnarray}
Then, the computation is divided into two cases:

{\bf Case $1$: $\ell_{0}\not\subset V, p:=\ell_{0}\cap V$}.
By \cite[Lemma 4.1]{Rez24a}, the pullback 
$i_{V}^{\ast}I_{\ell_{0}}\cong I_{p/V}$.
Thus, from \eqref{M2_wall_equ1}, we obtain
\begin{equation}\label{M2_wall_equ1-1}
\Ext^{1}(i_{V_{\ast}}I_{Z_{1}/V}^{\vee}(-4),I_{\ell_{0}}(-1))
\cong 
H^{0}(I_{p/V}\otimes I_{Z_{1}/V}(4)).
\end{equation}
The discussion is as follows.
\begin{enumerate}
\item[(i)] If $p=Z_{1}$, by \cite[Lemma 4.4]{Rez24a}, $I_{p/V}\otimes I_{p/V}$ fits into the short exact sequence
\begin{equation}\label{Ip_otimes_Ip_SES}
\xymatrix@C=0.5cm{
0 \ar[r] & \CO_{p} \ar[r] & I_{p/V}\otimes I_{p/V} \ar[r] & I_{p/V}^{2} \ar[r] & 0.
}
\end{equation}
Based on \eqref{M2_wall_equ1-1}, twisting $\CO_{V}(4)$,  
the exact sequence \eqref{Ip_otimes_Ip_SES} immediately yields
$$
\Ext^{1}(i_{V_{\ast}}I_{Z_{1}/V}^{\vee}(-4),I_{\ell_{0}}(-1))
\cong
H^{0}(\CO_{p})\oplus H^{0}(I^{2}_{p/V}(4))=\CN^{13},
$$
where $H^{0}(I^{2}_{p/V}(4))=\CN^{12}$ follows from Lemma \ref{k_points_P2} (2).
\item[(ii)]
If $p\neq Z_{1}$, 
by \cite[Lemma 4.4]{Rez24a}, we have $I_{p/V}\otimes I_{Z_{1}/V}\cong I_{p\cup Z_{1}/V}$.
Hence, according to \eqref{M2_wall_equ1-1} and Lemma \ref{k_points_P2} (1), we obtain 
$$
\Ext^{1}(i_{V_{\ast}}I_{Z_{1}/V}^{\vee}(-4),I_{\ell_{0}}(-1))
\cong 
H^{0}(I_{p\cup Z_{1}/V}(4))=\CN^{13}.
$$
\end{enumerate}
 
{\bf Case $2$: $\ell_{0}\subset V$}.
By \cite[Lemma 4.1]{Rez24a}, we get $i_{V}^{\ast}I_{\ell_{0}}\cong \CO_{V}(-1)\oplus \CO_{\ell_{0}}(-1)$.
Thus, from \eqref{M2_wall_equ1} and Lemma \ref{k_points_P2} (2), 
we derive
\begin{equation}\label{M2_wall_equ1-2}
\Ext^{1}(i_{V_{\ast}}I_{Z_{1}/V}^{\vee}(-4),I_{\ell_{0}}(-1))
\cong  
\CN^{9}\oplus H^{0}(I_{Z_{1}/V}\otimes \CO_{\ell_{0}}(3)).
\end{equation}
If $Z_{1}\in \ell_{0}$, by \cite[Lemma 4.4]{Rez24a}, 
we have  $I_{Z_{1}/V}\otimes \CO_{\ell_{0}}\cong \CO_{\ell_{0}}(-1)\oplus \CO_{Z_{1}}$.
If $Z_{1}\not\in \ell_{0}$, by \cite[Lemma 4.4]{Rez24a}, 
we have $I_{Z_{1}/V}\otimes \CO_{\ell_{0}} \cong \CO_{\ell_{0}}$.  
In both cases, we obtain $H^{0}(I_{Z_{1}/V}\otimes \CO_{\ell_{0}}(3))= \CN^{4}$. 
It follows from \eqref{M2_wall_equ1-2} that $\Ext^{1}(i_{V_{\ast}}I_{Z_{1}/V}^{\vee}(-4),I_{\ell_{0}}(-1))=\CN^{13}$.

For the last Ext group, by adjoint functors and Serre duality, we get 
\begin{eqnarray}\label{M2_wall_equ1-3}
\Ext^{1}(I_{\ell_{0}}(-1),i_{V_{\ast}}I_{Z_{1}/V}^{\vee}(-4))
&\cong &
\Ext^{1}(i_{V}^{\ast}I_{\ell_{0}}(-1),I_{Z_{1}/V}^{\vee}(-4)) \nonumber\\
&\cong &
H^{1}(i_{V}^{\ast}I_{\ell_{0}}\otimes I_{Z_{1}/V}).
\end{eqnarray}
Likewise, the computation is divided into two cases:

{\bf Case $1$: $\ell_{0}\not \subset V$, $p:=\ell_{0}\cap V$}.
We have $H^{1}(i_{V}^{\ast}I_{\ell_{0}}\otimes I_{Z_{1}/V})
\cong
H^{1}(I_{p/V}\otimes I_{Z_{1}/V})$. 
If $p=Z_{1}$, 
using the exact sequence \eqref{Ip_otimes_Ip_SES}, 
we get $H^{1}(I_{p/V}\otimes I_{Z_{1}/V}) =\CN^{2}$.
If $p\neq Z_{1}$, we deduce $H^{1}(I_{p/V}\otimes I_{Z_{1}/V})\cong H^{1}(I_{p\cup Z_{1}/V})= \CN$.

{\bf Case $2$: $\ell_{0}\subset V$}.
We have $H^{1}(i_{V}^{\ast}I_{\ell_{0}}\otimes I_{Z_{1}/V}) \cong H^{1}(I_{Z_{1}/V}(-1))\oplus H^{1}(I_{Z_{1}/V}\otimes \CO_{\ell_{0}}(-1))$. 
Considering cohomology of the short exact sequence 
$$
\xymatrix@C=0.5cm{
0 \ar[r] & I_{Z_{1}/V}(-1) \ar[r] & \CO_{V}(-1) \ar[r] & \CO_{Z_{1}} \ar[r] & 0,
}
$$
we obtain $H^{1}(I_{Z_{1}/V}(-1))= \CN$.
If $Z_{1}\in \ell_{0}$, 
then we have $H^{1}(I_{Z_{1}/V}\otimes \CO_{\ell_{0}}(-1))\cong H^{1}(\CO_{\ell_{0}}(-2))\oplus H^{1}(\CO_{Z_{1}})= \CN$.
If $Z_{1}\not\in \ell_{0}$, 
then we have $H^{1}(I_{Z_{1}/V}\otimes \CO_{\ell_{0}}(-1))=H^{1}(\CO_{\ell_{0}}(-1))=0$. 
In summary, the fourth statement follows form \eqref{M2_wall_equ1-3}. 
\end{proof}

The third wall sufficiently close to $\Gamma$ from the right side is given by the destabilizing pair $\langle \{\CO(-1)\xrightarrow{s} \CO_{\ell_{0}}\},\CO_{V}(-4)\rangle$.

\begin{prop}\label{M3_wall}
For the destabilizing pair $\langle \{\CO(-1)\xrightarrow{s} \CO_{\ell}\},\CO_{V}(-4)\rangle$, we have:
\begin{enumerate}
\item[(1)] $\Ext^{1}(\CO_{V}(-4),\CO_{V}(-4))=\CN^{3}$;
\item[(2)] $\Ext^{1}(\{\CO(-1)\xrightarrow{s} \CO_{\ell_{0}}\},\{\CO(-1)\xrightarrow{s} \CO_{\ell_{0}}\})=\CN^{5}$;
\item[(3)] 
$\Ext^{1}(\CO_{V}(-4),\{\CO(-1)\xrightarrow{s} \CO_{\ell_{0}}\})
=
\begin{cases}
\CN^{15} & \textrm{ if } \ell_{0} \not\subset V \textrm{ and } Z(s) = \ell_{0}\cap V;\\
\CN^{14} & \textrm{ if } \ell_{0} \not\subset V \textrm{ and } Z(s) \neq \ell_{0}\cap V;\\
\CN^{15} & \textrm{ if } \ell_{0} \subset V;
\end{cases}
$
\item[(4)] 
$\Ext^{1}(\{\CO(-1)\xrightarrow{s} \CO_{\ell_{0}}\}, \CO_{V}(-4))
= 
\begin{cases}
\CN &\textrm{ if } \ell_{0} \not\subset V \textrm{ and } Z(s) = \ell_{0}\cap V;\\
0 & \textrm{ if } \ell_{0} \not\subset V \textrm{ and } Z(s) \neq \ell_{0}\cap V;\\
\CN & \textrm{ if } \ell_{0} \subset V,
\end{cases}$
\end{enumerate}
where the point $Z(s)$ is the zero locus of $s$.
\end{prop}

\begin{proof}
The first Ext group is clear. For the second, see e.g. \cite[Lemma 4.5]{Rez24a}.
For the third statement, by Grothendieck--Verdier duality, we obtain
\begin{equation}\label{M3_wall_equ-1}
\Ext^{1}(\CO_{V}(-4),\{\CO(-1)\xrightarrow{s}\CO_{\ell_{0}}\}) 
\cong 
\Hom(\CO_{V}(-4),i_{V}^{\ast}\{\CO \xrightarrow{s}\CO_{\ell_{0}}(1)\}).
\end{equation}
The computation is divided into three cases:

{\bf Case $1$: $\ell_{0} \not\subset V$ and $Z(s)=\ell_{0}\cap V$}.
By \cite[Lemma 4.1]{Rez24a}, we have $i_{V}^{\ast}\{\CO\xrightarrow{s}\CO_{\ell_{0}}(1)\}
\cong
\CO_{V}\oplus\CO_{\ell_{0}\cap V}[-1]$.
Hence, from \eqref{M3_wall_equ-1}, we obtain 
$$\Ext^{1}(\CO_{V}(-4),\{\CO(-1)\xrightarrow{s}\CO_{\ell_{0}}\})\cong H^{0}(\CO_{V}(4))\oplus H^{0}(\CO_{\ell_{0}\cap V}[-1])=\CN^{15}.$$

{\bf Case $2$:  $\ell_{0} \not\subset V$ and $Z(s)\neq \ell_{0}\cap V$}. 
By \cite[Lemma 4.1]{Rez24a}, we see
$
i_{V}^{\ast}\{\CO\xrightarrow{s}\CO_{\ell_{0}}(1)\}
\cong
I_{\ell_{0}\cap V/V}
$.
Thus, from \eqref{M3_wall_equ-1} and Lemma \ref{k_points_P2} (2),
we get
$$
\Ext^{1}(\CO_{V}(-4),\{\CO(-1)\xrightarrow{s}\CO_{\ell_{0}}\})
\cong
H^{0}(I_{\ell_{0}\cap V/V}(4))=\CN^{14}.
$$

{\bf Case $3$: $\ell_{0}\subset V$}. 
By \cite[Lemma 4.1]{Rez24a}, we know
$
i_{V}^{\ast}\{\CO\xrightarrow{s}\CO_{\ell_{0}}(1)\}
\cong
\{\CO_{V} \xrightarrow{s} \CO_{\ell_{0}}(1)\}\oplus \CO_{\ell_{0}}.
$
It follows from \eqref{M3_wall_equ-1} that
\begin{equation}\label{M3_wall_equ-1-1}
\Ext^{1}(\CO_{V}(-4),\{\CO(-1)\xrightarrow{s}\CO_{\ell_{0}}\})
\cong 
H^{0}(\{\CO_{V}(4)\xrightarrow{s} \CO_{\ell_{0}}(5)\})\oplus \CN^{5}.
\end{equation}
Note that $\{\CO_{V}\xrightarrow{s} \CO_{\ell_{0}}(1)\}$ fits into the exact triangle
\begin{equation}\label{M3-wall-equ-1-2}
\xymatrix@C=0.5cm{
I_{\ell_{0}/V} \ar[r] & \{\CO_{V}\xrightarrow{s} \CO_{\ell_{0}}(1)\} \ar[r] & \CO_{Z(s)}[-1] \ar[r] & I_{\ell_{0}/V}[1].
}
\end{equation}
Since $H^{-1}(\CO_{Z(s)}[-1])=0=H^{0}(\CO_{Z(s)}[-1])$, 
taking cohomology of \eqref{M3-wall-equ-1-2} with twisting $\CO_{V}(4)$, we obtain
$$
H^{0}(\{\CO_{V}(4)\xrightarrow{s} \CO_{\ell_{0}}(5)\})\cong H^{0}(I_{\ell_{0}/V}(4))\cong H^{0}(\CO_{V}(3))=\CN^{10}.
$$
Thus, by \eqref{M3_wall_equ-1-1}, we conclude $\Ext^{1}(\CO_{V}(-4),\{\CO(-1)\xrightarrow{s}\CO_{\ell_{0}}\})=\CN^{15}$.

For the last Ext group, by adjoint functors and Serre duality, we get
\begin{eqnarray}\label{M3-wall-equ-2}
\Ext^{1}(\{\CO(-1)\xrightarrow{s} \CO_{\ell_{0}}\},\CO_{V}(-4))
&\cong &
\Ext^{1}(i_{V}^{\ast}\{\CO(-1)\xrightarrow{s} \CO_{\ell_{0}}\},\CO_{V}(-4))\nonumber\\
&\cong &
H^{1}(i_{V}^{\ast}\{\CO\xrightarrow{s} \CO_{\ell_{0}}(1)\}).  
\end{eqnarray}
Similarly,  
it is also divided into three cases:

{\bf Case $1$: $\ell_{0} \not\subset V$ and $Z(s)=\ell_{0}\cap V$}.
By \eqref{M3-wall-equ-2}, we get 
\begin{eqnarray*}
\Ext^{1}(\{\CO(-1)\xrightarrow{s} \CO_{\ell_{0}}\},\CO_{V}(-4))
&\cong&
H^{1}(\CO_{V})\oplus H^{1}(\CO_{\ell_{0}\cap V}[-1]) \\
&\cong& H^{0}(\CO_{\ell_{0}\cap V})=\CN.  
\end{eqnarray*}

{\bf Case $2$: $\ell_{0} \not\subset V$ and $Z(s)\neq \ell_{0} \cap V$}.
Let $\ell^{\prime}\in V$ be a line containing $\ell_{0}\cap V$.
Since $H^{1}(\CO_{V}(-1))=0=H^{1}(\CO_{\ell^{\prime}}(-1))$, consider the sheaf cohomology of 
$
\xymatrix@C=0.4cm{
0 \ar[r] & \CO_{V}(-1) \ar[r] & I_{\ell_{0}\cap V/V} \ar[r] & \CO_{\ell^{\prime}}(-1) \ar[r] & 0
}
$,
we derive $H^{1}(I_{\ell_{0}\cap V/V})=0$.
It follows from \eqref{M3-wall-equ-2} that 
$
\Ext^{1}(\{\CO(-1)\xrightarrow{s} \CO_{\ell_{0}}\},\CO_{V}(-4))\cong H^{1}(I_{\ell_{0}\cap V/V})=0.
$

{\bf Case $3$: $\ell_{0}\subset V$}.
Since $H^{1}(I_{\ell_{0}/V})=0=H^{2}(I_{\ell_{0}/V})$ and $H^{1}(\CO_{Z(s)}[-1])=\CN$,
considering the sheaf cohomology of \eqref{M3-wall-equ-1-2}, we derive
$H^{1}(\{\CO_{V}\xrightarrow{s}\CO_{\ell_{0}}(1)\})\cong H^{1}(\CO_{Z(s)}[-1])=\CN$. Therefore, 
from \eqref{M3-wall-equ-2}, we get
$$\Ext^{1}(\{\CO(-1)\xrightarrow{s} \CO_{\ell_{0}}\},\CO_{V}(-4))
\cong H^{1}(\{\CO_{V}\xrightarrow{s}\CO_{\ell_{0}}(1)\})\oplus H^{1}(\CO_{\ell_{0}})
=\CN.$$
This completes the proof of the proposition.
\end{proof}

The fourth wall sufficiently close to $\Gamma$ from the right side is given by the destabilizing pair $\langle \CO(-1),i_{V_{\ast}}I_{Z_{4}/V}^{\vee}(-5)\rangle$.

\begin{prop}\label{M4_wall}
For the destabilizing pair $\langle \CO(-1),i_{V_{\ast}}I_{Z_{4}/V}^{\vee}(-5)\rangle$, we have:
\begin{enumerate}
\item[(1)] $\Ext^{1}(\CO(-1),\CO(-1))=0$;
\item[(2)] $\Ext^{1}(i_{V_{\ast}} I_{Z_{4}/V}^{\vee}(-5),i_{V_{\ast}} I_{Z_{4}/V}^{\vee}(-5))=\CN^{11}$;
\item[(3)] $\Ext^{1}(i_{V_{\ast}}I_{Z_{4}/V}^{\vee}(-5),\CO(-1)) = \CN^{17}$; and
\item[(4)] $\Ext^{1}(\CO(-1),i_{V_{\ast}}I_{Z_{4}/V}^{\vee}(-5))
=\begin{cases}
\CN^{2} & \textrm{ if } Z_{4} \textrm{ lies in a line},\\
\CN & \textrm{otherwise.}
\end{cases}
$
\end{enumerate}
\end{prop}

\begin{proof}
The first two Ext groups are clear.
For the third statement, by Grothendieck--Verdier duality and Lemma \ref{k_points_P2} (1), we obtain
\begin{eqnarray*}
\Ext^{1}(i_{V_{\ast}}I_{Z_{4}/V}^{\vee}(-5),\CO(-1)) &\cong& \Ext^{1}(I_{Z_{4}/V}^{\vee}(-5),i_{V}^{\ast}\CO[-1])\\
& \cong & H^{0}(I_{Z_{4}/V}(5))=\CN^{17}.  
\end{eqnarray*}

For the last statement, by adjoint functors and Serre duality on $V$, we deduce
\begin{equation}\label{M4-wall-equ-2}
\Ext^{1}(\CO(-1),i_{V_{\ast}}I_{Z_{4}/V}^{\vee}(-5))\cong \Ext^{1}(\CO_{V}(-1),I_{Z_{4}/V}^{\vee}(-5))\cong H^{1}(I_{Z_{4}/V}(1)).
\end{equation}
Then, the computation is divided into two cases:

{\bf Case $1$: $Z_{4}$ lies in a line $\ell$.}
Then, there is a short exact sequence
\begin{equation}\label{Z4-sub-V-ex-seq}
\xymatrix@C=0.5cm{
0 \ar[r] & \CO_{V}(-1) \ar[r] & I_{Z_{4}/V} \ar[r] & \CO_{\ell}(-4) \ar[r] & 0.
}   
\end{equation}
Since $H^{1}(\CO_{\ell}(-3))\cong H^{0}(\CO_{\ell}(1))=\CN^{2}$,
considering the sheaf cohomology of the exact sequence \eqref{Z4-sub-V-ex-seq} with twisting $\CO_{V}(1)$, we obtain $H^{1}(I_{Z_{4}/V}(1))\cong H^{1}(\CO_{\ell}(-3)) = \CN^{2}$.
From \eqref{M4-wall-equ-2}, we deduce
$\Ext^{1}(\CO(-1),i_{V_{\ast}}I_{Z_{4}/V}^{\vee}(-5))=\CN^{2}$.

{\bf Case $2$: $Z_{4}$ does not lie in a line.}
In this case, we have $h^{0}(I_{Z_{4}/V}(1))=0$.
Since $h^{0}(\CO_{Z_{4}})=4$ and $h^{0}(\CO_{V}(1))=3$, taking cohomology of the structure sheaf sequence
$\xymatrix@C=0.4cm{
0 \ar[r] & I_{Z_{4}/V} \ar[r] & \CO_{V} \ar[r] & \CO_{Z_{4}} \ar[r] & 0
}$
with twisting $\CO_{V}(1)$,
we get 
$$
h^{1}(I_{Z_{4}/V}(1))=h^{0}(I_{Z_{4}/V}(1))+1=1.
$$
By \eqref{M4-wall-equ-2}, we conclude $\Ext^{1}(\CO(-1),i_{V_{\ast}}I_{Z_{4}}^{\vee}(-5))=\CN$. 
\end{proof}


\section{Proof of Theorem \ref{mainthm}}\label{proofs-main}

This section is devoted to the proof of Theorem \ref{mainthm}.
To the best of our knowledge, the complex manifold structure of the stability manifold of $\DC(\PB^{3})$ is still unknown.
In order to study the global geometry of the moduli space of stable pairs associated with quintic genus $2$ curves via wall-crossing in Bridgeland stability, we need to choose a special path that crosses well-structured chambers and connects the large volume limit chamber in the stability manifold.
To this end, we choose a continuous path $\gamma$ as in Figure \ref{main-wall-crossing-picture} crossing from the first wall to the last and being sufficiently close to the left branch $\Gamma$ of the hyperbola $\beta^{2}-\alpha^{2}=10$.
\begin{figure}[H]
\centering     
\begin{tikzpicture}[x=0.5pt,y=0.5pt,yscale=-0.8,xscale=0.8]
\draw    (35,338.5) -- (600,343.5) ; 
\draw [color={rgb, 255:red, 217; green, 143; blue, 22 }  ,draw opacity=1 ][line width=1.5]  [dash pattern={on 5.63pt off 4.5pt}]  (49,83.5) .. controls (75.91,118.29) and (149.96,336.36) .. (140,342.5) ;
\draw [color={rgb, 255:red, 208; green, 2; blue, 27 }  ,draw opacity=1 ][line width=1.5]    (128.39,284.03) .. controls (210,262.5) and (445,310.5) .. (481,329.5) ;
\draw [color={rgb, 255:red, 208; green, 2; blue, 27 }  ,draw opacity=1 ][line width=1.5]    (111.54,223.96) .. controls (191,221.5) and (447,281.5) .. (479,320.5) ;
\draw [color={rgb, 255:red, 208; green, 2; blue, 27 }  ,draw opacity=1 ][line width=1.5]    (111.54,223.96) .. controls (178,172.5) and (428,203.5) .. (477,250.5) ; 
\draw [color={rgb, 255:red, 208; green, 2; blue, 27 }  ,draw opacity=1 ][line width=1.5]    (83,148.5) .. controls (162.97,88.13) and (434.18,150.47) .. (478,181.5) ;
\draw [color={rgb, 255:red, 101; green, 184; blue, 6 }  ,draw opacity=1 ]   (154,289.78) .. controls (152.49,271.29) and (104.5,132.65) .. (76.83,93.65) ;
\draw [shift={(76,92.5)}, rotate = 53.6] [color={rgb, 255:red, 101; green, 184; blue, 6 }  ,draw opacity=1 ][line width=0.75]    (10.93,-3.29) .. controls (6.95,-1.4) and (3.31,-0.3) .. (0,0) .. controls (3.31,0.3) and (6.95,1.4) .. (10.93,3.29)   ;
\draw [color={rgb, 255:red, 0; green, 0; blue, 0 }  ,draw opacity=1 ]   (130,294.56) .. controls (175,288.17) and (227,324.93) .. (225,341.7) ;
\draw (174,246) node [anchor=north west][inner sep=0.75pt]  [font=\small,rotate=-358.73]  {$\displaystyle \MM_{1}$};
\draw (174,200) node [anchor=north west][inner sep=0.75pt]  [font=\small,rotate=-1.52]  {$\MM_{2} =\widetilde{\MM_{1}} \cup \MM_{2}^{\prime }$};
\draw (174,155) node [anchor=north west][inner sep=0.75pt]  [font=\small,rotate=-1.98]  {$\MM_{3} =\widetilde{\widetilde{\MM_{1}}} \cup \MM_{2}^{\prime } \cup \MM_{3}^{\prime } \cup \MM_{3}^{\prime\prime } \cup \MM_{3}^{\prime\prime\prime}$};
\draw (174,75) node [anchor=north west][inner sep=0.75pt]  [font=\small,rotate=-1.39]  {$\MM_{4} =\widetilde{\widetilde{\widetilde{\MM_{1}}}} \cup \widetilde{\MM_{2}^{\prime }} \cup \MM_{3}^{\prime } \cup \MM_{3}^{\prime\prime } \cup \widetilde{\MM_{3}^{\prime\prime\prime}} \cup \MM_{4}^{\prime }$};
\draw (174,283.52) node [anchor=north west][inner sep=0.75pt]  [font=\small]  {$\MM_{0}=\emptyset $};
\draw (539,37.49) node [anchor=north west][inner sep=0.75pt]   [align=left] {$ $};
\draw (110,114.19) node [anchor=north west][inner sep=0.75pt]   [align=left] {$ $};
\draw (165,320.04) node [anchor=north west][inner sep=0.75pt]   [align=left] {{\tiny BMT semicircle}};
\draw (110,170) node [anchor=north west][inner sep=0.75pt]  [font=\Large] [align=left] {$\displaystyle \gamma $};
\draw (84,226.03) node [anchor=north west][inner sep=0.75pt]  [font=\LARGE] [align=left] {$\displaystyle \Gamma $};
\end{tikzpicture}
\caption{Wall-crossing along a path sufficiently close to $\Gamma$}
\label{main-wall-crossing-picture}
\end{figure}
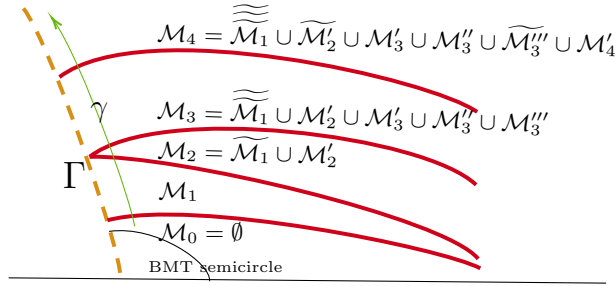

To begin with, let us consider the Bridgeland moduli space for the first chamber in Figure \ref{main-wall-crossing-picture}. 

\begin{prop}\label{moduli-chamber-1}
The Bridgeland  moduli space $\MM_{0}$ for the first chamber below the first wall, which is given by the destabilizing pair $\langle \CO(-2),B \rangle$ (or $\langle I_{C_{2}}(-1), \CO_{V}(-3) \rangle$), is empty.
Here, $B$ is a Gieseker semistable sheaf fitting into a short exact sequence
\eqref{B_SES}, $V$ is a plane and $C_{2}$ is a conic.
\end{prop}

\begin{proof}
The semicircle $(\beta+\frac{7}{2})^{2}+\alpha^2=(\frac{3}{2})^{2}$ defined by the pair $\langle \CO(-2),B\rangle$ is above the semicircle $(\beta+\frac{33}{10})^{2}+\alpha^{2}=(\frac{\sqrt{89}}{10})^{2}$ which is given by the BMT inequality in Theorem \ref{BMT-inequality}. Note that the moduli space is empty below this BMT semicircle. 
Since the moduli spaces are isomorphic in the same chamber, the moduli space $\MM_{0}$ is empty.
\end{proof}

Next, we study the Bridgeland moduli space $\mathcal{M}_{1}$ for the second chamber after the first wall as in Figure \ref{main-wall-crossing-picture}.
As mentioned above, the first wall is given by two different destabilizing pairs.
A prior, their relationship is unknown.
Hence, we need to consider the corresponding moduli spaces separately.
Suppose that $\MM_{1}^{\prime}\subset \MM_{1}$ is the moduli space parameterizing the extensions
\begin{equation}\label{1st-wall-singular-SES}
\xymatrix@C=0.5cm{
0 \ar[r] & I_{C_{2}}(-1) \ar[r] & E^{\prime} \ar[r] & \CO_{V}(-3) \ar[r] & 0,
}
\end{equation}
where $V$ is a plane and $C_{2}$ is a conic.
By Lemma \ref{M1_1_wall}, 
$\mathcal{M}_{1}^{\prime}$ is a $18$-dimensional $\mathbb{P}^{7}$-bundle over $(\PB^{3})^{\vee}\times \mathrm{Hilb}^{2t+1}(\PB^{3})$.
According to Lemma \ref{transverselem}, a generic element in $\MM_{1}^{\prime}$  is the ideal sheaf of the union of a plane cubic and a conic not in the plane. 
Let $\mathcal{K}_{4}(2,2)$ be the King moduli space of quiver representations of the Kronecker quiver $K_{4}$ with the dimension vector $(2,2)$. 
This moduli space $\mathcal{K}_{4}(2,2)$ is a normal irreducible projective variety of dimension $9$; moreover,
by \cite[Corollary 4]{BGLM24}, it is a projective variety of Fano type and thus a Mori dream space.

\begin{prop}\label{Modulisp1}
The moduli space $\mathcal{M}_{1}$ is a $20$-dimensional $\mathbb{P}^{11}$-bundle over  $\mathcal{K}_{4}(2,2)$, which 
generically parameterizes quintic genus $2$ curves. 
Moreover, the singular locus of $\MM_{1}$ is $\mathcal{M}_{1}^{\prime}$, which generically parameterizes the union of a plane cubic and a conic not in the plane. 
\end{prop}

\begin{proof}
Let $\MM_{1}^{\prime\prime}\subset\MM_{1}$ be the moduli space given by the extensions
$$
\xymatrix@C=0.5cm{
0 \ar[r] & \CO(-2) \ar[r] & E \ar[r] & B \ar[r] & 0,
}$$ 
where $B$ is a Gieseker semistable sheaf fitting into \eqref{B_SES}.
Note that the moduli space parameterizing $B$ in \eqref{B_SES} is isomorphic to $\mathcal{K}_{4}(2,2)$ and it is singular at $B$ strictly semistable (cf. the proof of \cite[Corollary 4.5]{Sch23}).
If $B$ is Gieseker stable, by \cite[Lemma 4.4]{GHS18}, we have
\begin{eqnarray}\label{M1-ext1-bound}
\ext^{1}(E,E)
&\leq& \ext^{1}(\CO(-2),\CO(-2))+\ext^{1}(B,\CO(-2))\nonumber\\
& + & \ext^{1}(\CO(-2),B)+\ext^{1}(B,B)-1 \nonumber\\
&=&20.  
\end{eqnarray}
By the Hizrebruch--Riemann--Roch theorem, we obtain $\chi(E,E)=-19$.
Using Serre duality and
comparing the tilt slopes of $E$ and $E(-4)$, we deduce $$\Ext^{3}(E,E)\cong\Hom(E,E(-4))=0.$$
It follows from \eqref{M1-ext1-bound} that $\Ext^{2}(E,E)=0$. Therefore, the moduli space $\MM_{1}^{\prime\prime}$ is smooth at $[E]$.
By Proposition \ref{M1_2_wall}, the moduli space $\mathcal{M}_{1}^{\prime\prime}$ is a $20$-dimensional $\mathbb{P}^{11}$-bundle over $\mathcal{K}_{4}(2,2)$, which is singular at $B$ strictly semistable.
Moreover, since every quintic genus $2$ curve is contained into a unique irreducible quadric surface (smooth or a quadric cone), so a generic element in $\MM_{1}^{\prime\prime}$ is the ideal sheaf of a quintic genus $2$ curve. 

Next, we claim $\MM_{1}\cong \MM_{1}^{\prime\prime}$. 
As a prior, we have $\MM_{1}=\MM_{1}^{\prime}\cup \MM_{1}^{\prime\prime}$.
Hence, it is sufficient to prove $\MM_{1}^{\prime}\subset \MM_{1}^{\prime\prime}$.
Suppose that $E^{\prime}\in \MM_{1}^{\prime}$ is represented by a non-trivial extension as \eqref{1st-wall-singular-SES}.
Let $V^{\prime}\subset \PB^{3}$ be the linear span of $C_{2}$. 
Since the composition of $I_{V^{\prime}}(-1)\rightarrow I_{C_{2}}(-1)$ and $I_{C_{2}}(-1)\rightarrow E^{\prime}$ is injective, by Theorem \ref{main-wall-in-Bridgeland-stability}, we have the short exact sequence of the second row in \eqref{M1-wall-commutative-diagram}
with $B^{\prime}$ fitting into \eqref{B_SES}.
As a result, we have a commutative diagram of  short exact sequences of the first two rows in \eqref{M1-wall-commutative-diagram}.
Then, by Snake Lemma, we obtain the following commutative diagram
\begin{equation}\label{M1-wall-commutative-diagram}
\xymatrix@C=0.5cm{
  I_{V^{\prime}}(-1) \ar[d]^{\textrm{id}} \ar[r] & I_{C_{2}}(-1) \ar[r] \ar[d] & \CO_{V^{\prime}}(-3)  \ar[d]  \\
 \CO(-2) \ar[r]  & E^{\prime} \ar[r] \ar[d] & B^{\prime}  \ar[d] \\
 & \CO_V(-3) \ar[r]^{\textrm{id}}  & \CO_V(-3). & 
}
\end{equation}
Hence, by definition, $E^{\prime}$ lies in $\MM_{1}^{\prime\prime}$.
This means $\MM_{1}^{\prime}\subset \MM_{1}^{\prime\prime}$ and thus $\MM_{1}\cong \MM_{1}^{\prime\prime}$.
Moreover, since the third column in \eqref{M1-wall-commutative-diagram} is a short exact sequence, so $B^{\prime}$ is strictly semistable.
As a result, we obtain that $\MM_{1}^{\prime}$ lies in the singular locus of $\MM_{1}$.

Now, we show that the singular locus of $\MM_{1}$ is indeed $\MM_{1}^{\prime}$.
Suppose that $E\in \mathcal{M}_{1}$ is a singular point defined by the non-trivial extension
\begin{equation}\label{1st-wall-SES-main}
\xymatrix@C=0.5cm{
0 \ar[r] & \CO(-2) \ar[r] & E \ar[r] & B \ar[r] & 0,
}   
\end{equation}
where $B$ is strictly Gieseker semistable.
We claim that
$B$ fits into a short exact sequence
\begin{equation}\label{1st-wall-B-SES-strictly}
\xymatrix@C=0.5cm{
0 \ar[r] & \CO_{V^{\prime}}(-3) \ar[r] & B \ar[r] & \CO_{V}(-3) \ar[r] & 0,
}   
\end{equation}
where $V$ and $V^{\prime}$ are two planes in $\PB^{3}$. 
In fact, since $B$ is strictly semistable,
there exists a short exact sequence of Gieseker semistable sheaves
\begin{equation*}\label{1st-wall-B-abstract-SES}
\xymatrix@C=0.5cm{
0 \ar[r] & F \ar[r] & B \ar[r] & G \ar[r] & 0
}  
\end{equation*}
such that their slopes in Gieseker stability are same.
Since $B$ is a torsion sheaf, so $F$ and $G$ are also torsion sheaves.
By comparing the reduced Hilbert polynomial of $F$ and $B$, we get 
$\ch(F)=(0,1,-\frac{7}{2},\frac{37}{6})=\ch(G)$. 
Note that $\ch(F(3))=\ch(G(3))=(0,1,-\frac{1}{2},\frac{1}{6})=\ch(\CO_{V^{\prime\prime}})$ for any plane $V^{\prime\prime}$.
It follows that $F\cong \CO_{V^{\prime}}(-3)$ and $G\cong \CO_{V}(-3)$, where $V$ and $V^{\prime}$ are planes.
Note that the composition of $E\rightarrow B$ in \eqref{1st-wall-SES-main} and $B \rightarrow \CO_{V}(-3)$ in \eqref{1st-wall-B-SES-strictly} is non-trivial. 
By Theorem \ref{main-wall-in-Bridgeland-stability}, we get $E\in \MM_{1}^{\prime}$.
As a consequence, $\MM_{1}^{\prime}$ is the singular locus of $\MM_{1}$.
\end{proof}

To determine which irreducible component survives in the large volume limit, we apply the same proof as that of \cite[Lemma 6.1]{Rez24a}.
Theorem \ref{main-wall-in-Bridgeland-stability} implies
the following:

\begin{lem}\label{wallcorssing-lem}
If the $0$-th cohomology object of a general object in an irreducible component created by a wall is an ideal sheaf, then this irreducible component (birationally) survives all the way up to the large volume limit.
\end{lem}

Let $\mathcal{M}_{2}$ be the Bridgeland moduli space for the third chamber after the second wall which is given by the destabilizing pair $\langle I_{\ell}(-1),i_{V_{\ast}}I_{Z_{1}/V}^{\vee}(-4) \rangle$. Let $\mathfrak{Fl}_{i}$ be the space parameterizing flag $Z_{i}\subset V\subset \PB^{3}$, where $V$ is a plane and $Z_{i}$ is a $0$-dimensional subscheme of length $i$.

\begin{prop}\label{Modulisp2}
The moduli space $\mathcal{M}_{2}$ has two irreducible components: $\widetilde{\mathcal{M}_{1}}$ and $\mathcal{M}_{2}^{\prime}$.
The component $\widetilde{\mathcal{M}_{1}}$ is birational to $\mathcal{M}_{1}$.
The component $\mathcal{M}_{2}^{\prime}$ is a $21$-dimensional  $\mathbb{P}^{12}$-bundle over $\mathrm{Gr}(2,4)\times \mathfrak{Fl}_{1}$, which generically parameterizes the union of a line and a plane quartic curve together with a choice of one point on the quartic.
\end{prop}

\begin{proof}
By Proposition \ref{Modulisp1}, the $0$-th cohomology sheaf of a general object in $\mathcal{M}_{1}$ is an ideal sheaf.
By Lemma \ref{wallcorssing-lem}, crossing the second wall $\langle I_{\ell}(-1),i_{V_{\ast}}I_{Z_{1}/V}^{\vee}(-4) \rangle$, there exists an irreducible component $\widetilde{\MM_{1}}$ in $\MM_{2}$, which is birational to $\MM_{1}$.
Let $\MM_{2}^{\prime}\subset \MM_{2}$ be the moduli space parameterizing the non-trivial extensions
\begin{equation}\label{2nd-wall-ex-seq}
\xymatrix@C=0.5cm{
0 \ar[r] & I_{\ell}(-1) \ar[r] & E \ar[r] & i_{V_{\ast}}I_{Z_{1}/V}^{\vee}(-4) \ar[r] & 0.
}
\end{equation}
By Proposition \ref{M2_wall}, we know that $\mathcal{M}_{2}^{\prime}$ is a $\mathbb{P}^{12}$-bundle over $\mathrm{Gr}(2,4)\times \mathfrak{Fl}_{1}$. Since $\dim(\mathcal{M}_{2}^{\prime})=12+4+5=21>20=\dim(\widetilde{\mathcal{M}_{1}})$, so $\MM_{2}^{\prime}$ is an irreducible component of $\MM_{2}$.

Next, taking cohomology sheaves of \eqref{2nd-wall-ex-seq} and using Lemma \ref{IZm-duality}, we obtain $\CH^{1}(E)\cong \CH^{1}(i_{V_{\ast}}I_{Z_{1}/V}^{\vee}(-4))\cong \CO_{Z_{1}}$ and a short exact sequence
$$
\xymatrix@C=0.5cm{
0 \ar[r] & I_{\ell}(-1) \ar[r] & \CH^{0}(E) \ar[r] & \CO_{V}(-4) \ar[r] & 0.
}
$$
If $\ell\not\subset V$, by Lemma \ref{transverselem} to $D=\ell$, it follows that $\CH^{0}(E)$ is an ideal sheaf $I_{\ell\cup C_{4}}$, where $C_{4}\subset V$ is a plane quartic. 
Moreover, 
by Lemma \ref{wallcorssing-lem}, a generic object
$E$ must be a stable pair. Therefore,
$Z_{1}$ lies in the plane quartic curve $C_{4}$ but not in $\ell$.
\end{proof}

The following lemma shows that the irreducible component $\MM_{2}^{\prime}$ entirely crosses the third wall. 
Let $\mathcal{N}_{3}$ be the locus $E\in \Ext^{1}(\{\CO(-1)\xrightarrow{s}\CO_{\ell}\},\CO_{V}(-4))$.
By Proposition \ref{Modulisp2}, we have $\mathcal{N}_{3} \subset\widetilde{\MM_{1}}$.

\begin{lem}\label{cross_M2pr}
The locus $\mathcal{N}_{3}$  intersects $\MM_{2}^{\prime}$ empty.
\end{lem}

\begin{proof}
Assume that there exists $E\in \Ext^{1}(\{\CO(-1)\xrightarrow{s}\CO_{\ell}\},\CO_{V}(-4))$ such that $E\in \MM_{2}^{\prime}$. 
Taking cohomology sheaves of 
$$
\xymatrix@C=0.5cm{
0 \ar[r] & \CO_{V}(-4) \ar[r] & E \ar[r] & \{\CO(-1)\xrightarrow{s}\CO_{\ell}\} \ar[r] & 0,
}
$$ we obtain $\CH^{1}(E)=\CO_{Z(s)}$ and $\CH^{0}(E)$ fits into the short exact sequence
\begin{equation}\label{M3-wall-H0E-SES-1}
\xymatrix@C=0.5cm{
0 \ar[r] & \CO_{V}(-4) \ar[r] & \CH^{0}(E) \ar[r] & I_{\ell}(-1) \ar[r] & 0.
}
\end{equation}
Next, taking cohomology sheaves of \eqref{2nd-wall-ex-seq},
we obtain $\CH^{1}(E)=\CO_{Z_{1}}=\CO_{Z(s)}$ and $\CH^{0}(E)$ fits into the short exact sequence 
\begin{equation}\label{M3-wall-H0E-SES-2}
\xymatrix@C=0.5cm{
0 \ar[r] & I_{\ell}(-1) \ar[r] & \CH^{0}(E) \ar[r] & \CO_{V}(-4) \ar[r] & 0.
}
\end{equation}
Since $\Hom(I_{\ell}(-1),I_{\ell}(-1))=\CN$,
by comparing \eqref{M3-wall-H0E-SES-1} and \eqref{M3-wall-H0E-SES-2}, the exact sequence \eqref{M3-wall-H0E-SES-1} splits contradicting the tilt stability of $\CH^{0}(E)$.
\end{proof}

Let $\mathcal{M}_{3}$ be the Bridgeland moduli space for the fourth chamber after the third wall which is given by the destabilizing pair $\langle \{\CO(-1)\xrightarrow{s}\CO_{\ell}\},\CO_{V}(-4) \rangle$. Let $\mathcal{U}$ be the universal line over $\mathrm{Gr}(2,4)$ and $\mathrm{Fl}(4)$ the flag manifold of $\CN^{4}$.

\begin{prop}\label{Modulisp3}
The moduli space $\mathcal{M}_{3}$ has five irreducible components: $\widetilde{\widetilde{\mathcal{M}_{1}}}$, ${\mathcal{M}_{2}^{\prime}}$, $\mathcal{M}_{3}^{\prime}$,
$\mathcal{M}_{3}^{\prime\prime}$
and
$\mathcal{M}_{3}^{\prime\prime\prime}$.
The first two are birational or isomorphic to their counterparts in $\mathcal{M}_{2}$ and the three new components:
\begin{enumerate}
\item $\mathcal{M}_{3}^{\prime}$ is a $21$-dimensional $\mathbb{P}^{14}$-bundle over $\mathrm{Gr}(2,4)\times (\PB^{3})^{\vee}$, which generically parameterizes the disjoint union of a line and a plane quartic;
\item $\mathcal{M}_{3}^{\prime\prime}$ is a $21$-dimensional $\mathbb{P}^{13}$-bundle over $\mathcal{U}\times(\PB^{3})^{\vee}$, which generically parameterizes the union of a line  and a plane quartic intersecting the line;
\item  $\mathcal{M}_{3}^{\prime\prime\prime}$ is a $20$-dimensional  $\mathbb{P}^{14}$-bundle over $\mathrm{Fl}(4)$, which generically parameterizes the union of a plane cubic with a thickening of a line in the plane.
\end{enumerate}

\end{prop}

\begin{proof}
By Lemma \ref{cross_M2pr}, the component $\MM_{2}^{\prime}$ is an irreducible component of $\MM_{3}$.
According to Proposition \ref{Modulisp2}, the $0$-th cohomology sheaf of a general object in $\widetilde{\MM_{1}}$ is an ideal sheaf.
By Lemma \ref{wallcorssing-lem}, crossing the third wall, it follows that there exists an irreducible component $\widetilde{\widetilde{\mathcal{M}_{1}}}$ in $\MM_{3}$, which is birational to $\widetilde{\MM_{1}}$.
Using Proposition \ref{M3_wall}, we obtain three new irreducible components $\MM_{3}^{\prime}$, $\MM_{3}^{\prime\prime}$ and $\MM_{3}^{\prime\prime\prime}$: a $\mathbb{P}^{14}$-bundle over $\mathrm{Gr}(2,4)\times (\PB^{3})^{\vee}$, a $\mathbb{P}^{13}$-bundle over $\mathcal{U}\times(\PB^{3})^{\vee}$ and a $\mathbb{P}^{14}$-bundle over $\mathrm{Fl}(4)$, respectively. 

Next, we will describe the general objects in $\MM_{3}^{\prime}$, $\MM_{3}^{\prime\prime}$ and $\MM_{3}^{\prime\prime\prime}$, respectively.
Denote $A:=\{\CO(-1)\xrightarrow{s}\CO_{\ell}\}$ and $B:=\CO_{V}(-4)$.
Note that $A$ fits into an exact triangle
\begin{equation}\label{A-eactseq}
\xymatrix@C=0.5cm{
I_{\ell}(-1) \ar[r] & A \ar[r] & \CO_{p}[-1] \ar[r] & I_{\ell}(-1)[1],
}
\end{equation}
where $p=Z(s)\in \ell$ is the zero locus of $s$. 
Since $\Ext^{2}(\CO_{V}(-4),I_{\ell}(-1))=0$,
applying $\Hom(B,-)$ to \eqref{A-eactseq},
we have a short exact sequence
\begin{equation}\label{AB-extseq}
\xymatrix@C=0.4cm{
0 \ar[r] & \Ext^{1}(\CO_{V}(-4),I_{\ell}(-1)) \ar[r] & \Ext^{1}(B,A) \ar[r]^{\psi\;\;\;\;\;\;\;} & \Hom(\CO_{V}(-4),\CO_{p}) \ar[r] & 0.
}    
\end{equation}
For any $E\in \Ext^{1}(B,A)$, we have a non-trivial extension
\begin{equation}\label{3nd-wall-ex-seq}
\xymatrix@C=0.5cm{
0 \ar[r] & A \ar[r] & E \ar[r] & B \ar[r] & 0.
}   
\end{equation}
To obtain the description of $\mathcal{H}^{0}(E)$ and $\mathcal{H}^{1}(E)$, it is sufficient to study the image of $\psi$ in \eqref{AB-extseq}.
Taking cohomology sheaves of \eqref{3nd-wall-ex-seq},
we obtain a long exact sequence of sheaves
\begin{equation}\label{E-extsq}
\xymatrix@C=0.5cm{
  0 \ar[r] & \mathcal{I}_{\ell}(-1) \ar[r]  & \mathcal{H}^0(E) \ar[r] & \CO_{V}(-4) \ar[r]^{\;\;\;\;\;\varphi} & \CO_{p} \ar[r] & \mathcal{H}^{1}(E) \ar[r] & 0.
}   
\end{equation}
In \eqref{E-extsq}, if $\varphi=0$, then $\mathcal{H}^{1}(E)=\CO_{p}$ and $\CH^{0}(E)$ fits into a short exact sequence 
\begin{equation}\label{H0E-extsq}
\xymatrix@C=0.5cm{
0 \ar[r] & I_{\ell}(-1) \ar[r] & \mathcal{H}^{0}(E) \ar[r] & \CO_{V}(-4) \ar[r] & 0.
}  
\end{equation}
If $\varphi\neq 0$, then $\varphi$ is surjective. Hence, $\CH^{1}(E)=0$ and $E$ fits into a short exact sequence
\begin{equation}\label{H0E-extsq2}
\xymatrix@C=0.5cm{
0 \ar[r] & I_{\ell}(-1) \ar[r] & E \ar[r] & I_{p/V}(-4) \ar[r] & 0.
}
\end{equation}
In the following, by Proposition \ref{M3_wall}, we will discuss three new components case-by-case:

{\bf Case (1): $\ell \not \subset V$ and $\ell \cap V =p$.}
In \eqref{E-extsq}, if $\varphi=0$, then $\mathcal{H}^{1}(E)=\CO_{p}$ and $\CH^{0}(E)$ fits into \eqref{H0E-extsq}.
By Grothendieck--Verdier duality and \cite[Lemma 4.1]{Rez24a}, we have
$$
\Ext^{1}(\CO_{V}(-4),I_{\ell}(-1))
\cong
\Ext^{1}(\CO_{V}(-4),i_{V}^{\ast}I_{\ell}[-1])
\cong \Hom(\CO_{V},I_{p}(4))=\CN^{14}.
$$
This corresponds to a closed subscheme which is a $\mathbb{P}^{13}$-bundle over $\mathrm{Gr}(2,4)\times (\PB^{3})^{\vee}$. 
Thus, in \eqref{E-extsq}, the morphism $\varphi$ is non-trivial for a general element $E\in \Ext^{1}(B,A)$.
This means
$\CH^{1}(E)=0$ and $E$ fits into \eqref{H0E-extsq2}.
By Grothendieck--Verdier duality and \cite[Lemma 4.1]{Rez24a}, we deduce
\begin{eqnarray*}
\Ext^{1}(I_{p/V}(-4),I_{\ell}(-1)) &\cong & \Ext^{1}(I_{p/V}(-4),i_{V}^{\ast} I_{\ell}[-1])\\
&\cong &  \Hom(I_{p/V},I_{p/V}(4))\cong H^{0}(\CO_{V}(4)).    
\end{eqnarray*}
It follows that for a general class, the sheaf $E$ is the ideal sheaf of the disjoint union of a line and a plane quartic.

{\bf Case (2): $\ell \not \subset V$ and $\ell\cap V \neq p$.} 
Since the zero locus $p \in \ell $ of $s$, so $p\not \in V$ and thus $\varphi=0$ in \eqref{E-extsq}. 
Hence, we get $\mathcal{H}^{1}(E)\cong \CO_{p}$ and $\mathcal{H}^{0}(E)$ fits into \eqref{H0E-extsq}.
By Lemma \ref{transverselem}, we know that $\mathcal{H}^{0}(E)$ is the ideal sheaf of the union of a line and a plane quartic. The extension class $\Ext^{1}(\CO_{V}(-4),I_{\ell}(-1))\cong H^{0}(I_{\ell\cap V}(4))$ implies that the plane quartic contains the intersection point $\ell\cap V$.

{\bf Case (3): $\ell\subset V$.}
In \eqref{E-extsq}, if $\varphi=0$, then $\mathcal{H}^{1}(E)=\CO_{p}$ and $\CH^{0}(E)$ fits into \eqref{H0E-extsq}.
By Grothendieck--Verdier duality and \cite[Lemma 4.1]{Rez24a}, we get 
\begin{eqnarray*}
\Ext^{1}(\CO_{V}(-4),I_{\ell}(-1)) &\cong & \Ext^{1}(\CO_{V}(-4),i_{V}^{\ast}I_{\ell}[-1])\\
& \cong & \Hom(\CO_{V},\CO_{V}(3)\oplus\CO_{\ell}(3))=\CN^{14}.
\end{eqnarray*}
This corresponds to a closed subscheme which is a $\mathbb{P}^{13}$-bundle over $\mathrm{Fl}(4)$.
Thus, the morphism $\varphi$ is non-trivial for the general class in $E\in \Ext^{1}(B,A)$. This means 
$\CH^{1}(E)=0$ and $E$ fits into \eqref{H0E-extsq2}.

Now let $\mathbb{L}$ be the double line obtained by thickening $\ell$, with tangent direction of infinitesimal thickening contained in the plane $V$ at $(\ell\cap C_{3})\cup p$, where $C_{3}$ is a plane cubic in $V$.
Since $\mathbb{L}\cup C_{3} \subset \mathbb{L}\cup V$, so we have a short exact sequence
$$
\xymatrix@C=0.5cm{
0 \ar[r] & I_{\ell}(-1) \ar[r] & I_{\mathbb{L}\cup C_{3}} \ar[r] & I_{p/V}(-4) \ar[r] & 0.
}
$$
Applying $\Hom(-,\CO_{V}(-4))$ to the above exact sequence, we obtain a non-trivial composition $\rho: E^{\prime}:=I_{\mathbb{L}\cup C_{3}}\twoheadrightarrow I_{p/V}(-4) \hookrightarrow \CO_{V}(-4)$.
Let $\mathrm{Cone}(\rho)$ be the cone of $E^{\prime} \xrightarrow{\rho} \CO_{V}(-4)$.
Then we have an exact triangle
\begin{equation}\label{Cone-rho-ext-tr}
\xymatrix@C=0.5cm{
\mathrm{Cone}(\rho)[-1] \ar[r] & E^{\prime} \ar[r] & \CO_{V}(-4).
}
\end{equation}
According to the octahedral axiom, we have the following commutative diagram of exact triangles
$$
\xymatrix@C=0.5cm{
I_{\mathbb{L}\cup C_{3}} \ar[r] \ar[d]^{id} & I_{p/V}(-4) \ar[r] \ar@{->}[d] & I_{\ell}(-1)[1]  \ar@{->}[d] \\
I_{\mathbb{L}\cup C_{3}} \ar[r]^{\rho}  & \CO_{V}(-4) \ar[r] \ar@{->}[d] & \mathrm{Cone}(\rho) \ar@{->}[d]
\\
& \CO_{p} \ar[r]^{id} & \CO_{p}. 
}
$$ 
Since $p\in \ell\subset V$, by Grothendieck--Verdier duality and \cite[Lemma 4.1]{Rez24a}, we get
$$
\Ext^{1}(\CO_{p}[-1],I_{\ell}(-1))
\cong \Ext^{1}(\CO_{p},\CO_{V}(-1)\oplus\CO_{\ell}(-1))=\CN.
$$
By taking cohomology sheaves of the last column,
it follows that $\CH^{0}(\mathrm{Cone}(\rho)[-1])\cong I_{\ell}(-1)$ and $\CH^{1}(\mathrm{Cone}(\rho)[-1])\cong \CO_{p}$.
By \eqref{Cone-rho-ext-tr}, we obtain $\Hom(\CO_{q},\mathrm{Cone}(\rho))=0$ for all points $q\in \PB^{3}$.
According to Lemma \ref{stab-pair-equ-def}, $\mathrm{Cone}(\rho)[-1]\cong \{\CO(-1)\xrightarrow{s}\CO_{\ell}\}:=A$.
As a consequence, from \eqref{Cone-rho-ext-tr}, $E^{\prime}$ is a class in $\Ext^{1}(B,A)$.
Hence, a generic class $E\in \Ext^{1}(B,A)$ is the ideal sheaf of a plane cubic with a thickening of the line.
\end{proof}

The following lemma demonstrates that only irreducible components $\MM_{3}^{\prime}$ and $\MM_{3}^{\prime\prime}$ entirely cross the fourth wall.

\begin{lem}\label{wallcross_M4}
Let $\mathcal{N}_{4}$ be the locus $E\in \Ext^{1}(\CO(-1),i_{V_{\ast}}I_{Z_{4}/V}^{\vee}(-5))$.
Then $\mathcal{N}_{4}$
intersects both $\MM_{2}^{\prime}$ and $\MM_{3}^{\prime\prime\prime}$ non-empty, 
and intersects both $\MM_{3}^{\prime}$ and $\MM_{3}^{\prime\prime}$ empty.
\end{lem}

\begin{proof}
Note that any object $E\in\mathcal{N}_{4}$ is given by a non-trivial extension
$$
\xymatrix@C=0.5cm{
0 \ar[r] & i_{V_{\ast}}I_{Z_{4}/V}^{\vee}(-5) \ar[r] & E \ar[r] & \CO(-1) \ar[r] & 0.
}
$$ 
Taking cohomology sheaves, Lemma \ref{IZm-duality} implies the following exact sequence
\begin{equation}\label{under-4-th-wall}
\xymatrix@C=0.5cm{
0 \ar[r] & \CO_{V}(-5) \ar[r] & \CH^{0}(E) \ar[r] & \CO(-1) \ar[r]^{\,\,\,\psi} & \CO_{Z_{4}} \ar[r] & \CH^{1}(E) \ar[r] & 0.
}
\end{equation}

For the first statement, we assume in \eqref{under-4-th-wall} that there are three points
$Z_{3}\subset Z_{4}$ lying in a line $ \ell\subset V$ and $\CH^{1}(E)\cong \CO_{q}$, where $q\subset Z_{4}$ is another point.
Hence, we have a short exact sequence
\begin{equation}\label{under-4th-wall-special-case-SES}
\xymatrix@C=0.5cm{
0 \ar[r] & \CO_{V}(-5) \ar[r] & \CH^{0}(E) \ar[r] & I_{Z_{3}}(-1) \ar[r] & 0.
}   
\end{equation}
Since $\Ext^{i}(I_{\ell}(-1),\CO_{V}(-5))=0$ for $i=0,1$, applying $\Hom(I_{\ell}(-1),-)$ to \eqref{under-4th-wall-special-case-SES},
we get 
$\Hom(I_{\ell}(-1),\CH^{0}(E))\cong \Hom(I_{\ell}(-1),I_{Z_{3}}(-1))$.
Therefore, combining Snake Lemma, we obtain the following commutative  diagram of six short exact sequences of sheaves
$$
\xymatrix@C=0.5cm{
0 \ar[r] \ar[d] & I_{\ell}(-1) \ar[r] \ar[d] & I_{\ell}(-1) \ar[d] \\
\CO_{V}(-5) \ar[r] \ar[d] & \CH^{0}(E) \ar[r] \ar[d] & I_{Z_{3}}(-1) \ar[d]\\
\CO_{V}(-5) \ar[r] & Q \ar[r] & \CO_{\ell}(-4) 
}
$$
By direct computations, we have $\Ext^{1}(\CO_{\ell}(-4),\CO_{V}(-5))=\CN$. 
It follows that $Q\cong \CO_{V}(-4)$ or $Q\cong \CO_{V}(-5)\oplus \CO_{\ell}(-4)$.
Since $\CH^{0}(E)$ is tilt stable, so we derive $Q\cong \CO_{V}(-4)$.
By the definitions of $\MM_{3}^{\prime\prime\prime}$ (assuming $q\in \ell$) and $\MM_{2}^{\prime}$, 
we obtain the first statement.

For the second statement, assume that $\mathcal{N}_{4}$ intersects both $\MM_{3}^{\prime}$ and $\MM_{3}^{\prime\prime}$  non-empty.
Then, following the notation in the proof of Proposition \ref{Modulisp3},
we have two cases:
\begin{enumerate}
\item In \eqref{E-extsq}, the morphism $\varphi$ is surjective. 
This implies 
$\CH^{1}(E)=0$ and $
\ker(\varphi)\cong I_{p/V}(-4).
$
Hence, in \eqref{under-4-th-wall}, we have $\ker(\psi)\cong I_{Z_{4}}(-1)$.
By \eqref{under-4-th-wall}, we obtain an injective morphism $I_{\ell}(-1)\rightarrow I_{Z_{4}}(-1)$.
It follows that $Z_{4}\subset \ell \subset V$. 
\item In \eqref{E-extsq}, the morphism $\varphi=0$. 
This implies $\CH^{1}(E)\cong \CO_{p}$. 
Hence, by \eqref{under-4-th-wall}, we get an injective morphism $I_{\ell}(-1)\rightarrow I_{Z_{3}}(-1)$, where $Z_{3}=Z_{4}\setminus p$.
It follows that $Z_{3}\subset \ell$.
Thus, we get $Z_{4}\subset \ell\subset V$.
\end{enumerate}
In both cases, it contradicts the definitions of $\MM_{3}^{\prime}$ and $\MM_{3}^{\prime\prime}$.
This finishes the proof of the lemma.
\end{proof}

Let $\mathcal{M}_{4}$ be the Bridgeland moduli space for the fifth chamber after the fourth wall $\langle \CO(-1),i_{V_{\ast}}I_{Z_{4}/V}^{\vee}(-5) \rangle$.

\begin{prop}\label{Modulisp4}
The moduli space $\MM_{4}$ has six irreducible components: $\widetilde{\widetilde{\widetilde{\mathcal{M}_{1}}}}$,
$\widetilde{\mathcal{M}_{2}^{\prime}}$,
$\mathcal{M}_{3}^{\prime}$,
${\mathcal{M}_{3}^{\prime\prime}}$,
$\widetilde{\mathcal{M}_{3}^{\prime\prime\prime}}$
and $\mathcal{M}_{4}^{\prime}$. The first five are birational or isomorphic to their counterparts in $\mathcal{M}_{3}$. The last new component is a $27$-dimensional  $\mathbb{P}^{16}$-bundle over $\mathfrak{Fl}_{4}$, which generically parameterizes a plane quintic curve together with a choice of four points on it.
\end{prop}

\begin{proof}
By Lemma \ref{wallcross_M4}, the components $\MM_{3}^{\prime}$ and $\MM_{3}^{\prime\prime}$ are irreducible components of $\MM_{4}$.
According to Proposition  $\ref{Modulisp3}$, 
the $0$-th cohomology sheaf of a general object is an ideal sheaf in $\widetilde{\widetilde{\mathcal{M}_{1}}}$ (resp. $\mathcal{M}_{2}^{\prime}$ or $\mathcal{M}_{3}^{\prime\prime\prime}$). 
By Lemma \ref{wallcorssing-lem}, crossing the fourth wall, it follows that there exist three irreducible components $\widetilde{\widetilde{\widetilde{\mathcal{M}_{1}}}}$,
${\widetilde{\mathcal{M}_{2}^{\prime}}}$, and
$\widetilde{\mathcal{M}_{3}^{\prime\prime\prime}}$ coming from $\mathcal{M}_{3}$, which are birational to $\widetilde{\widetilde{\mathcal{M}_{1}}}$,  $\mathcal{M}_{2}^{\prime}$ and $\mathcal{M}_{3}^{\prime\prime\prime}$, respectively.
By Lemma \ref{M4_wall}, there is a new irreducible component $\MM_{4}^{\prime}$ which is a $27$-dimensional $\mathbb{P}^{16}$-bundle over $\mathfrak{Fl}_{4}$.
Next, we will give a description of the generic objects in $\mathcal{M}_{4}^{\prime}$.
For any non-trivial extension $E\in \Ext^{1}(i_{V_{\ast}}I_{Z_{4}/V}^{\vee}(-5),\CO(-1))$, taking cohomology sheaves and using Lemma \ref{IZm-duality},
we obtain that $\CH^{0}(E)$ fits into the short exact sequence 
$$
\xymatrix@C=0.5cm{
0 \ar[r] & \CO(-1) \ar[r] & \CH^{0}(E) \ar[r] & \CO_{V}(-5) \ar[r] & 0
}
$$
and $\CH^{1}(E)\cong \CH^{1}(i_{V_{\ast}}I_{Z_{4}/V}^{\vee}(-5))\cong \CO_{Z_{4}}$. 
Applying Lemma \ref{transverselem} to $D=\emptyset$, we get $\CH^{0}(E)$ is an ideal sheaf of a plane quintic curve and $Z_{4}$ lies in the curve.
This completes the proof of Proposition \ref{Modulisp4}.
\end{proof}

Based on the above study of wall-crossings for the Bridgeland moduli spaces of semistable objects with class $v=(1,0,-5,11)$, 
we are now ready to finish the proofs of Theorem \ref{mainthm} and Corollary \ref{main-cor}.

\begin{proof}[Proof of Theorem \ref{mainthm}]
By Proposition \ref{Bri-mod=PT-mod}, the moduli space $P_{-1}(\PB^{3},5[\ell])$ is isomorphic to the Bridgeland moduli space $\MM_{4}$. 
As a result, Theorem \ref{mainthm} follows from Propositions \ref{Modulisp1}, \ref{Modulisp2}, \ref{Modulisp3} and \ref{Modulisp4}.
\end{proof}

\begin{proof}[Proof of Corollary \ref{main-cor}]
By Theorem \ref{mainthm}, there are six irreducible components of the moduli space of stable pairs for the quintic genus $2$ curves. 
Crossing the left branch $\Gamma$ of the hyperbola $\mu_{\alpha,\beta}(v)=0$, from the right side to the left side, the heart of $t$-structure changes from $\Coh^{\alpha,\beta}(\PB^{3})$ to $\Coh^{\alpha,\beta}(\PB^{3})[-1]$.
Let $I$ be the ideal sheaf of a generic curve in one of the components (i), (ii), (iii) and (iv) as in Corollary \ref{main-cor}.
Likewise to \cite[Lemma 9.2]{Rez24a}, as an actual wall, the left branch $\Gamma$ is given by $\langle I_{C_{i}},T_{i}[-1]\rangle$, where $C_{i}$ is a Cohen-Macaulay curve of degree $5$ and arithmetic genus $2+i$ with $1\leq i \leq 4$ and $T_{i}$ denotes a $0$-dimensional sheaf of length $i$.
For any generic curve in one of the components (i), (ii), (iii) and (iv), it is not contained in $C_{i}$.
It follows that $\Hom(I_{C_{i}},I)=0$.
Therefore, the above four components (birationally) survive after crossing $\Gamma$.
\end{proof}


\appendix

\section{Some technical Lemmas}\label{three-tech-lem}

To compute Ext groups associated with the second, third and fourth walls, we need the following result.

\begin{lem}\label{k_points_P2}
Let $k\geq 2$ be an integer.
\begin{enumerate}
\item[(1)] 
If $Z_{k}\subset \PB^{2}$ is a $0$-dimensional subscheme of length $k$ and $i\geq 2k$, then $h^{0}(I_{Z_{k}}(i-3))=\frac{(i-1)(i-2)}{2}-k$.
\item[(2)] If $p\in \PB^{2}$ is a point, then
$h^{0}(I_{p}(k))=\frac{k(k+3)}{2}$
and
$h^{0}(I_{p}^{2}(k))=\frac{k(k+3)}{2}-2$.
\end{enumerate}
\end{lem}

\begin{proof}
We only prove the first statement, 
the proof of the second statement is the same.
For (1), we claim $\Ext^{1}(I_{Z_{k}},\CO_{\PB^{2}}(-i))=0$.
Then, by Serre duality, we have
$$
H^{1}(I_{Z_{k}}(i-3))\cong\Ext^{1}(I_{Z_{k}},\CO_{\PB^{2}}(-i))=0.
$$
Consider cohomology of the structure exact sequence  
$$
\xymatrix@C=0.5cm{
0 \ar[r] & I_{Z_{k}}(i-3) \ar[r] & \CO_{\PB^{2}}(i-3) \ar[r] & \CO_{Z_{k}} \ar[r] & 0.
}
$$
Since $h^{0}(\CO(i-3))=\frac{(i-1)(i-2)}{2}$ and $h^{0}(\CO_{Z_{k}})=k$, it follows that $h^{0}(I_{Z_{k}}(i-3))=\frac{(i-1)(i-2)}{2}-k$.

Now, we prove the claim. By direct computations, we get $\ch^{\beta}(I_{Z_{k}})=(1,-\beta,\frac{\beta^{2}}{2}-k)$ and $\ch^{\beta}(\CO_{\mathbb{P}^{2}}(-i)[1])=(-1,\beta+i,-\frac{\beta^2}{2}-\beta i-\frac{i^{2}}{2})$. 
Since $i\geq 2k\geq 4$, so the numerical wall  
$$
W(I_{Z_{k}},\CO_{\mathbb{P}^{2}}(-i)[1])=\{(\alpha,\beta)\in \RN_{>0} \times \RN| \alpha^{2}+(\beta+\frac{2k+i^{2}}{2i})^{2}=(\frac{2k-i^{2}}{2i})^{2} \}
$$
is non-trivial. We discuss two cases.

{\bf Case (1): $i>2k$}.
In this case, the numerical wall $W(I_{Z_{k}},\CO_{\mathbb{P}^{2}}(-i)[1])$ intersects with $\beta=-1$.
Since $\ch^{-1}_{1}(I_{Z_{k}})=1$, so $I_{Z_{k}}$ has no walls along $\beta=-1$. 
As $I_{Z_{k}}$ is a slope stable sheaf, we obtain that
$I_{Z_{k}}$ is tilt stable along $\beta =-1$ and thus tilt stable along $W(I_{Z_{k}},\CO_{\mathbb{P}^{2}}(-i)[1])$.
Note that $\CO_{\mathbb{P}^{2}}(-i)[1]$ is also tilt stable for $\beta\geq -i$.
Hence, by comparing the tilt slopes, we have $\Ext^{1}(I_{Z_{k}},\CO_{\mathbb{P}^{2}}(-i))\cong \Hom(I_{Z_{k}},\CO_{\mathbb{P}^{2}}(-i)[1])=0$.

{\bf Case (2): $i=2k$}.
In this case, the semicircle $W(I_{Z_{k}},\CO_{\mathbb{P}^{2}}(-2k)[1])$ tangents to $\beta=-1$ and it is the biggest possible wall for $I_{Z_{k}}$. The discussion is divided into two cases.

(i) If $Z_{k}$ lies in a line $\ell$, then there is a short exact sequence
\begin{equation}\label{Zk-in-line-ex-seq}
\xymatrix@C=0.5cm{
0 \ar[r] & \CO_{\PB^{2}}(-1) \ar[r] & I_{Z_{k}} \ar[r] & \CO_{\ell}(-k) \ar[r] & 0.
} 
\end{equation}
Applying $\Hom(\CO_{\PB^{2}}(-2k+3),-)$ to \eqref{Zk-in-line-ex-seq}, we get $\Ext^{1}(\CO_{\PB^{2}}(-2k+3),I_{Z_{k}})=0$. 
By Serre duality, it follows that $\Ext^{1}(I_{Z_{k}},\CO_{\PB^{2}}(-2k))\cong \Ext^{1}(\CO_{\PB^{2}}(-2k+3),I_{Z_{k}})=0$. 

(ii) If $Z_{k}$ does not lie in a line, then $W(I_{Z_{k}},\CO_{\mathbb{P}^{2}}(-2k)[1])$ is not an actual wall.
In fact, the semicircle $W(I_{Z_{k}},\CO_{\mathbb{P}^{2}}(-2k)[1])$ intersects with $\beta=-2$.
Assume that the numerical wall $W(I_{Z_{k}},\CO_{\mathbb{P}^{2}}(-2k)[1])$ is an actual wall and the destabilized sequence is the short exact sequence 
$$
\xymatrix@C=0.5cm{
0 \ar[r] & A \ar[r] & I_{Z_{k}} \ar[r] & B \ar[r] & 0
}
$$
with $\ch^{-2}(A)=(a,b,\frac{c}{2})$, where $a,b,c\in \ZN$. 
Note that $\ch_{1}^{-2}(I_{Z_{k}})=2$. 
It follows that $b=1$.
By changing the roles of $A$ and $B$, we can assume $a\geq 1$.
According to $\mu_{\sqrt{2k-2},-2}(A)=\mu_{\sqrt{2k-2},-2}(E)$, we deduce $c=a(2k-2)+3-2k\geq 1$. 
By $0\leq \Delta(A)\leq \Delta(I_{Z_{k}})$, we obtain $1-2k\leq ac \leq 1$. 
This implies $a=c=1$. 
Hence, we get
$\ch(A)=(1,-1,\frac{1}{2})$ and $A\cong \CO_{\PB^{2}}(-1)$. 
Thus, $Z_{k}$ must lie in a line, a contradiction.
As a result, $I_{Z_{k}}$ is tilt stable along $W(I_{Z_{k}},\CO_{\mathbb{P}^{2}}(-2k)[1])$. 
By comparing the tilt slopes, we conclude $\Ext^{1}(I_{Z_{k}},\CO_{\mathbb{P}^{2}}(-2k))\cong \Hom(I_{Z_{k}},\CO_{\mathbb{P}^{2}}(-2k)[1])=0$.
This completes the proof of the lemma.
\end{proof}

The following two lemmas are frequently used in the proof of Theorem \ref{mainthm}.

\begin{lem}\label{IZm-duality}
Let $i_{V}: V\hookrightarrow  \PB^{3}$ be a plane and $Z_{m}\subset V$ a $0$-dimensional subshceme of length
$m\geq 1$. 
Then, for any integer $k\geq 0$, we have $\CH^{0}(i_{V_{\ast}}I_{Z_{m}/V}^{\vee}(-k))\cong \CO_{V}(-k)$ and 
$\CH^{1}(i_{V_{\ast}}I_{Z_{m}/V}^{\vee}(-k))\cong \CO_{Z_{m}}$.
\end{lem}

\begin{proof}
Consider the exact triangle
$
\xymatrix@C=0.4cm{
I_{Z_{m}/V} \ar[r] & \CO_{V} \ar[r] & \CO_{Z_{m}} \ar[r] & I_{Z_{m}/V}[1]
}
$.
Applying $\RCH(-,\CO_{V})$,
we obtain an exact triangle
\begin{equation}\label{Zm-extsq}
\xymatrix@C=0.5cm{
\RCH(\CO_{Z_{m}},\CO_{V}) \ar[r] & \CO_{V} \ar[r] & I_{Z_{m}/V}^{\vee} \ar[r] & \RCH(\CO_{Z_{m}},\CO_{V})[1].}  
\end{equation}
By Grothendieck--Verdier duality, we obtain $\RCH(\CO_{Z_{m}},\CO_{V})\cong \CO_{Z_{m}}[-2]$.
Tensoring \eqref{Zm-extsq} by $\CO_{V}(-k)$ and applying $i_{V_{\ast}}$, we get
an exact triangle
$$
\xymatrix@C=0.5cm{
i_{V_{\ast}}\CO_{Z_{m}}[-2] \ar[r] & i_{V_{\ast}}\CO_{V}(-k) \ar[r] & i_{V_{\ast}}I_{Z_{m}/V}^{\vee}(-k) \ar[r] & i_{V_{\ast}}\CO_{Z_{m}}[-1]
}
$$
in $\DC(\PB^{3})$.
Then,
taking cohomology sheaves, the lemma holds.
\end{proof}

\begin{lem}[{\cite[Lemma 6.2]{Rez24a}}]\label{transverselem}
Let $D\subset \PB^{3}$ be a closed subscheme of dimension $\leq 1$, $V\subset \PB^{3}$ a plane, $k$ a positive integer and $\#(D\cap V)<+\infty$. 
If $E$ fits into a non-trivial extension
\begin{equation}\label{E-ideal}
\xymatrix@C=0.5cm{
0 \ar[r] & I_{D}(-1) \ar[r] & E \ar[r] & \CO_{V}(-k) \ar[r] & 0
},    
\end{equation}
then $E$ is an ideal sheaf of a curve $D\cup C_{k}$, where $C_{k}\subset V$ is a curve of degree $k$.
\end{lem}

\begin{proof}
Let $i_{V}:V\hookrightarrow \PB^{3}$ be the closed inclusion.
First, we claim $i_{V}^{\ast}I_{D}\cong I_{D\cap V/V}$. Since $D\cap V\subset D \subset \PB^{3}$, so we have the following short exact sequence 
\begin{equation}\label{DcapV-SES}
\xymatrix@C=0.5cm{
0 \ar[r] & I_{D} \ar[r] & I_{D\cap V} \ar[r] & I_{D\cap V/D} \ar[r] & 0.
}
\end{equation}
Note that $i_{V}^{\ast}I_{D\cap V/D}=0$ and $i_{V}^{\ast}I_{D\cap V}\cong I_{D\cap V/V}$.
Now applying $i_{V}^{\ast}$ to \eqref{DcapV-SES}, we obtain $i_{V}^{\ast}I_{D}\cong i_{V}^{\ast}I_{D\cap V}\cong I_{D\cap V/V}$.

Suppose now that $E$ is not a torsion free sheaf. 
Note that any subsheaf of $\CO_{V}(-k)$ contains a subsheaf of the form $\CO_{V}(-i)$ for some $k<i$.
Let $T$ be the maximal torsion sheaf of $E$. By Snake Lemma, $T$ is a subsheaf of $\CO_{V}(-k)$. Hence, $\CO_{V}(-i)$ is a subsheaf of $E$.
Let $F:=E/\CO_{V}(-i)$ be the quotient sheaf. 
Since $\Hom(\CO_{V}(-i),I_{D}(-1))=0$, the composition $f$ of the morphism $I_{D}(-1)\rightarrow E$ in \eqref{E-ideal} and the quotient map $E\rightarrow F$ is non-trivial. 
Since $\ker(f)=\CO_{V}(-i)\cap I_{D}(-1)$ in $E$, so $\ker(f)$ is trivial.
It follows that $f$ is injective.
By Snake Lemma, we have the following commutative diagram of four short exact sequences
\begin{equation}\label{Ck-commutative-diagram-2}
\xymatrix@C=0.5cm{
& \CO_{V}(-i) \ar[r]^{id} \ar@{^(->}[d] & \CO_{V}(-i) \ar@{^(->}[d] \\
I_{D}(-1) \ar[d]^{id} \ar[r] & E \ar@{->>}[d] \ar[r] & \CO_{V}(-k) \ar@{->>}[d] \\
 I_{D}(-1) \ar[r]^{\,\,\,\,\,\,\;f} & F \ar[r] & G .
}
\end{equation}
If $\Ext^{1}(G,I_{D}(-1))= 0$, then the middle horizontal sequence in \eqref{Ck-commutative-diagram-2} splits . 
Hence, we see $\Ext^{1}(G,I_{D}(-1))\neq 0$. Applying $\Hom(-,I_{D}(-1))$ to the last vertical sequence in \eqref{Ck-commutative-diagram-2}, we obtain an exact sequence
\begin{equation}\label{G_HomID}
\xymatrix@C=0.26cm{
0 \ar[r] & \Ext^{1}(G,I_{D}(-1)) \ar[r] & \Ext^{1}(\CO_{V}(-k),I_{D}(-1)) \ar[r] & \Ext^{1}(\CO_{V}(-i),I_{D}(-1)). }
\end{equation}
By Grothendieck--Verdier duality, we have 
$$\Ext^{1}(i_{V_{\ast}}\CO_{V}(-j),I_{D}(-1))\cong
\Hom(\CO_{V},I_{D\cap V/V}(j)).$$
Since $I_{D\cap V/V}(k)\hookrightarrow I_{D\cap V/V}(i)$ is injective for $k<i$, so
the third morphism is injective in \eqref{G_HomID}, a contradiction with $\Ext^{1}(G,I_{D}(-1))\neq 0$. Hence, $E$ is a torsion free sheaf and thus an ideal sheaf of a curve.

Set $E:= I_{C}$, where
 $C$ is a curve.
By Snake Lemma, we have the following commutative diagram of short exact sequences
$$
\xymatrix@C=0.5cm{
I_{D}(-1) \ar[r] \ar[d]^{id} & I_{C} \ar[r] \ar@{^(->}[d] & \CO_{V}(-k) \ar@{^(->}[d]\\
I_{D}(-1) \ar[r] & \CO \ar[r] \ar@{->>}[d] & \CO_{D\cup V^{\prime}} \ar@{->>}[d]\\
& \CO_{C} \ar[r]^{id} & \CO_{C}, 
}
$$
where $V^{\prime}\subset\PB^{3}$ is a plane. 
In the last column, we see $D\cup V^{\prime}=V\cup C$. Therefore, $V=V^{\prime}$ and $D\subset C$ and $C=D\cup (C\cap V)$. Denote $C_{k}:=C\cap V$. 
From \eqref{E-ideal}, we derive 
$\ch_{2}(I_{C})=-k+\ch_{2}(I_{D})$ and thus
$C_{k}\subset V$ is a plane curve of degree $k$.
\end{proof}


\end{document}